\documentclass{siamart0516}

\usepackage{amsfonts}
\usepackage{multirow}
\usepackage{float}
\usepackage{array}
\usepackage[labelfont=bf]{caption}
\usepackage{textcomp}
\usepackage{amsmath}
\usepackage{amssymb}
\usepackage{algorithmic}
\usepackage{mathtools}
\usepackage{enumitem}

\setitemize{noitemsep,topsep=0pt,parsep=0pt,partopsep=0pt}

\DeclareMathOperator{\argmin}{argmin}

\newcommand{\R}{\mathbb{R}}

\newcommand{\inner}[2]{\langle{#1},{#2}\rangle}

\newcommand{\norm}[1]{\|#1\|}

\newcommand{\dom}{{\mbox{\rm dom\,}}}

\newcommand{\vgap}{\vspace{.1in}}

\newcommand{\tx}{\tilde x}

\newcommand{\lam}{\lambda}

\newcommand{\tv}{\tilde v}

\newcommand{\bConv}[1]{\overline{\mbox{\rm Conv}}\,(\Re^{#1})}

\newcommand{\bi}{\begin{itemize}}
\newcommand{\ei}{\end{itemize}}
\newcommand{\ba}{\begin{array}}
\newcommand{\ea}{\end{array}}
\newcommand{\gap}{\hspace*{2em}}

\usepackage{lipsum}
\usepackage{amsfonts}
\usepackage{graphicx}
\usepackage{epstopdf}
\usepackage{algorithmic}
\ifpdf
  \DeclareGraphicsExtensions{.eps,.pdf,.png,.jpg}
\else
  \DeclareGraphicsExtensions{.eps}
\fi

\newcommand{\TheTitle}{Complexity of a quadratic penalty accelerated  \\
inexact proximal point method for solving linearly constrained nonconvex composite programs}

\headers{Quadratic penalty accelerated method}{Weiwei Kong, J.G. Melo and R.D.C. Monteiro}

\title{{\TheTitle}}

\author{
 Weiwei Kong \thanks{School of Industrial and Systems
    Engineering, Georgia Institute of
    Technology, Atlanta, GA, 30332-0205. (E-mails: {\tt
       wkong37@gatech.edu} and  {\tt monteiro@isye.gatech.edu}). The work of Renato D.C. Monteiro
    was partially supported by NSF Grant CMMI-1300221, ONR Grant N00014-18-1-2077 and CNPq Grant 406250/2013-8.}
    \and
      Jefferson G. Melo \thanks{Institute of Mathematics and Statistics, Federal University of Goias, Campus II- Caixa
    Postal 131, CEP 74001-970, Goi\^ania-GO, Brazil. (E-mail:  {\tt jefferson@ufg.br}). The work of this author was
    supported in part by  CNPq Grants 406250/2013-8,  406975/2016-7  and FAPEG/GO.}
    \and
    Renato D.C. Monteiro \footnotemark[1]
}

\usepackage{amsopn}

\ifpdf
\hypersetup{
  pdftitle={Quadratic Penalty Method},
  pdfauthor={WeiWei}
}
\fi

\begin{document}

\maketitle

\begin{abstract}
This paper  analyzes the iteration-complexity of a quadratic penalty accelerated inexact proximal point method  for solving linearly constrained nonconvex  composite programs. More specifically, the objective function is of the form $f+h$ where $f$ is a differentiable function whose gradient is Lipschitz continuous and $h$ is a closed convex function with possibly unbounded domain.  The method, basically, consists of applying an accelerated inexact proximal point method for solving approximately a sequence of quadratic penalized subproblems associated to the linearly constrained problem. Each  subproblem of the proximal point method is in turn approximately solved by an accelerated composite gradient (ACG) method.  It is shown that
the proposed scheme generates a $\rho-$approximate stationary point  in at most $\mathcal{O}(\rho^{-3})$ ACG iterations.
Finally, numerical results  showing the efficiency of the proposed method are also given.
\end{abstract}



\section{Introduction}\label{sec:int}
Our main goal in this paper is to describe and establish the iteration-complexity of a quadratic penalty
accelerated inexact proximal point (QP-AIPP) method  for solving the
linearly constrained nonconvex composite minimization problem
\begin{equation}\label{eq:probintro}
\min \left\{ f(z) + h(z) : Az=b, \,  z \in \Re^n \right \}
\end{equation}
where $A \in \Re^{l\times n}$,  $b \in \Re^l$, $h:\Re^n \to (-\infty,\infty]$ is a proper lower-semicontinuous convex function 
and $f$ is a real-valued differentiable (possibly nonconvex) function whose gradient  is $L_f$-Lipschitz continuous on  $\dom h$.
For given tolerances $\hat \rho>0$ and $\hat \eta >0$, the main result of this paper shows that the 
QP-AIPP method, started from any point in $\dom h$ (but not necessarily satisfying $Az=b$), obtains a triple $(\hat z, \hat v, \hat p)$ satisfying
\begin{equation}\label{eq:approxOptimCond_Intro}
\hat v \in \nabla f(\hat z) + \partial h(\hat z) + A^* \hat p, \quad \|\hat v\| \le \hat \rho, \quad \|A \hat z-b\| \le \hat \eta
\end{equation}
in at most ${\cal O}(\hat \rho^{-2} \hat \eta^{-1})$ accelerated composite gradient (ACG) iterations.  It is worth noting that this result is obtained
under the mild assumption that the optimal value of \eqref{eq:probintro} is finite and hence assumes neither that $\dom h$ is bounded nor that
\eqref{eq:probintro} has an optimal solution.

The QP-AIPP  method is based on solving penalized subproblems
of the form
\begin{equation}\label{eq:PenaltyProb-intro}
\min \left\{ f(z) + h(z) + \frac{c}2 \| Az-b\|^2 : z \in \Re^n \right \}
\end{equation}
for an increasing sequence of positive penalty parameters $c$.
These subproblems in turn are approximately solved so as to satisfy the first two
conditions in \eqref{eq:approxOptimCond_Intro} and the QP-AIPP method terminates when $c$ is large enough so as to guarantee that
the third condition in \eqref{eq:approxOptimCond_Intro} also hold.
Moreover, each subproblem in turn is approximately solved by an accelerated inexact proximal point (AIPP)
method which solves a sequence of prox subproblems of the form
\begin{equation}\label{eq:penPbRegIntro}
\min \left\{ f(z) + h(z) + \frac{c}2 \| Az-b\|^2 + \frac{1}{2\lam}\|z-z_{k-1}\|^2 : z \in \Re^n \right \}
\end{equation}
where $z_{k-1}$ is the previous iterate and the next one, namely $z_k$, is a suitable approximate solution of \eqref{eq:penPbRegIntro}.
Choosing $\lambda$ sufficiently small ensures that the objective function of \eqref{eq:penPbRegIntro} is a convex composite optimization
which is approximately solved by an ACG method.

More generally, the AIPP method mentioned above solves problems of the form
\begin{equation}\label{eq:PenProb2Intro}
\phi_*:=\min \left\{ \phi(z):=g(z) + h(z)  : z \in \Re^n \right \}
\end{equation}
where $h$ is as above and $g$ is a differentiable function whose gradient is $M$-Lipschitz continuous on $\dom h$
and whose lower curvature is bounded below on $\dom h$ by some constant $m \in (0,M]$,  i.e.,
\[
 g(u)-\left[g(z)+\left\langle \nabla g(z),u-z\right\rangle \right] \ge - \frac{m}{2}\|u-z\|^{2}
\quad \forall \, z,u \in \dom h.
\]
Note that the penalized subproblem \eqref{eq:PenaltyProb-intro} is a special case of \eqref{eq:PenProb2Intro} with
$g(z) = f(z) +  (c/2) \| Az-b\|^2$, and hence $m=L_f$ and $M=L_f+c\|A\|^2$.
It is well-known that the composite gradient method finds a $\rho$-solution of \eqref{eq:PenProb2Intro},
i.e., a pair $(\bar{z},\bar{v})\in \dom h\times \Re^n$ such that $\bar{v}\in \nabla f(\bar{z})+\partial h(\bar{z})$ and $\|\bar{v}\| \leq \rho$, in at most ${\cal O}(M(\phi(z_0)-\phi_*)/\rho^2)$ composite-type iterations where  $z_0$ is the initial point.
On the other hand, 
the AIPP method finds such solution in at most
\begin{equation}\label{eq:boundIntro}
{\cal O}\left ( \frac{\sqrt{Mm}}{{\rho^2}}\min\left\{\phi(z_0)-\phi_*,md_0^2\right\} +  \sqrt{\frac{M}{m}}  \log^{+}_1\left(\frac{M}{m}\right) \right)
\end{equation}
composite type-iterations
where $d_0$ denotes the distance of $z_0$ to the set of optimal solutions of \eqref{eq:PenProb2Intro}.
Hence, its complexity is better than that for the composite gradient method by a factor of $\sqrt{M/m}$. 
The main advantage of the AIPP method is that its iteration-complexity bound has a lower
dependence on $M$, i.e., it is ${\cal O}(\sqrt{M})$ instead of the ${\cal O}(M)$-dependence of the
composite gradient method. Hence, the use of  the AIPP method instead of the
composite gradient method to solve \eqref{eq:penPbRegIntro} (whose associated $M={\cal O}(c)$) in the scheme outlined above
is  both theoretically and computationally appealing.

\vgap

{\it Related works.} 
Under the assumption that domain of $\phi$ is bounded, \cite{nonconv_lan16} presents an ACG method applied directly to \eqref{eq:PenProb2Intro}
 which obtains a $\rho$-approximate solution
of \eqref{eq:PenProb2Intro} in
\begin{equation} \label{eq:compl-lan}
{\cal O}\left( \frac{ MmD_h^2}{{ \rho}^2} + \left(\frac{Md_0}{\rho}\right)^{2/3}\right)
\end{equation}
where $D_h$ denotes the diameter of the domain of $h$.
Motivated by \cite{nonconv_lan16}, other papers 
have proposed ACG methods for solving \eqref{eq:PenProb2Intro}
under different assumptions on the functions $g$ and $h$
(see for example \cite{Aaronetal2017,Paquette2017,LanUniformly,Li_Lin2015,rohtua}).
 In particular, their analyses exploit the lower curvature $m$ and the work \cite{Aaronetal2017}, which assumes $h=0$, establishes
a complexity which depends on $\sqrt{M} \log M$ instead of $M$ as in \cite{nonconv_lan16}.  As in the latter work,
our AIPP method  also uses the idea of solving a sequence of convex proximal subproblems by an ACG
method, but solves them in a more relaxed manner and, as a result, achieves the complexity
bound \eqref{eq:boundIntro} which improves the one in \cite{Aaronetal2017} by a factor
of $\log( M/\rho)$. It should be noted that the second complexity bound in \eqref{eq:boundIntro} in terms of $d_0$ is new
in the context of the composite nonconvex problem \eqref{eq:PenProb2Intro} and follows as a special case of a more general bound, namely \eqref{auxcomplex00},
which actually unifies both bounds in \eqref{eq:boundIntro}. Moreover, in contrast to the analysis of \cite{nonconv_lan16},
ours does not assume that $D_h$ is finite.
Also, 
inexact proximal point methods and HPE variants of the ones studied in \cite{monteiro2010complexity,Sol-Sv:hy.ext}
for solving convex-concave saddle point problems and monotone variational inequalities,
which inexactly solve a sequence of proximal suproblems by means of an ACG variant, were previously proposed by \cite{YheMoneiroNash,YHe2,OliverMonteiro,MonteiroSvaiterAcceleration,LanADMM}.
The behavior of an accelerated gradient method near saddle points of unconstrained instances of \eqref{eq:PenProb2Intro} (i.e., with $h=0$) is studied in \cite{Jwright2017}.

Finally, complexity analysis of first-order quadratic penalty methods for solving special convex instances of \eqref{eq:probintro}
where $h$ is an indicator function
was first studied in \cite{LanRen2013PenMet} and further analyzed in
\cite{Aybatpenalty,molinari2019alternating,IterComplConicprog}.
Papers \cite{LanMonteiroAugLag,Patrascu2017} study the iteration-complexity of first-order augmented Lagrangian methods for solving
the latter class of convex problems.
The authors are not aware of earlier papers dealing with complexity analysis of quadratic penalty methods
for solving nonconvex constrained optimization problems.
However, \cite{ProxAugLag_Ming} studies the complexity of a proximal augmented Lagrangian method
for solving nonconvex instances of \eqref{eq:probintro}
under the very strong assumption that $\nabla f$ is Lipschitz continuous everywhere and $h=0$.

%
%
%



{\it Organization of the paper.}  Subsection~\ref{sec:DefNot} contains basic definitions and notation used in the paper.
Section~{2} is divided into two subsections. The first one introduces the composite nonconvex optimization problem  and discusses some approximate solutions criteria. The second subsection is devoted to the study of a general inexact proximal point framework to solve nonconvex optimization problems. In this subsection, we also show that a composite gradient method can be seen as an instance of the latter framework.  Section~\ref{sec:accgradmet} is divided into two subsections. The first one reviews an ACG method and its properties. Subsection~\ref{sec:AIPPmet} presents  the AIPP method and  its iteration-complexity analysis. Section~\ref{sec:penaltymet} states and analyzes the  QP-AIPP method for solving linearly constrained nonconvex composite  optimization problems. 
Section~\ref{sec:computResults} presents computational results. Section \ref{sec:remarks} gives some concluding remarks.
Finally, the appendix  gives the proofs of some technical results needed in our presentation.

\subsection{Basic definitions and notation} \label{sec:DefNot}

This subsection provides some basic definitions and notation used
in this paper.

The set of real numbers is denoted by $\Re$. The set of non-negative real numbers  and 
the set of positive real numbers are denoted by $\Re_+$ and $\Re_{++}$, respectively. We let $\Re^2_{++}:=\Re_{++}\times \Re_{++}$.
Let $\Re^n$ denote the standard $n$-dimensional Euclidean 
space with  inner product and norm denoted by $\left\langle \cdot,\cdot\right\rangle $
and $\|\cdot\|$, respectively. For $t>0$, define 
$\log^+_1(t):= \max\{\log t ,1\}$. The diameter of a set $D \subset \Re^n$ is defined as $\sup \{\|z-z'\|:z,z'\in D\}$.

Let $\psi: \Re^n\rightarrow (-\infty,+\infty]$ be given. The effective domain of $\psi$ is denoted by
$\dom \psi:=\{x \in \Re^n: \psi (x) <\infty\}$ and $\psi$ is proper if $\dom \psi \ne \emptyset$.
Moreover, a proper function $\psi: \Re^n\rightarrow (-\infty,+\infty]$ is $\mu$-strongly convex for some $\mu \ge 0$ if
$$
\psi(\alpha z+(1-\alpha) u)\leq \alpha \psi(z)+(1-\alpha)\psi(u) - \frac{\alpha(1-\alpha) \mu}{2}\|z-u\|^2
$$
for every $z, u \in \dom \psi$ and $\alpha \in [0,1]$.
If $\psi$ is differentiable at $\bar z \in \Re^n$, then its affine   approximation $\ell_\psi(\cdot;\bar z)$ at $\bar z$ is defined as
\begin{equation}\label{eq:defell}
\ell_\psi(z;\bar z) :=  \psi(\bar z) + \inner{\nabla \psi(\bar z)}{z-\bar z} \quad \forall  z \in \Re^n.
\end{equation}
Also, for $\varepsilon \ge 0$,  its \emph{$\varepsilon$-subdifferential} at $z \in \dom \psi$ is denoted by
\begin{equation}\label{eq:epsubdiff}
\partial_\varepsilon \psi (z):=\left\{ v\in\Re^n: \psi(u)\geq \psi(z)+\left\langle v,u-z\right\rangle -\varepsilon,\forall u\in\Re^n\right\}.
\end{equation}
The subdifferential of $\psi$ at $z \in \dom \psi$, denoted by $\partial \psi (z)$, corresponds to  $\partial_0 \psi(z)$.
The set of all proper lower semi-continuous convex functions $\psi:\Re^n\rightarrow (-\infty,+\infty]$  is denoted by $\bConv{n}$.

The proof of the following result can be found in \cite[Proposition 4.2.2]{Hiriart2}.

\begin{proposition}\label{prop:transpForm}  Let $\psi:\Re^n\rightarrow (-\infty,+\infty]$, $z, \bar z \in \dom \psi$ and $v \in \Re^n$ be given and
assume  that $v\in\partial \psi (z)$. Then,
$v\in \partial_\varepsilon \psi (\bar{z})$
where $\varepsilon =  \psi(\bar{z})-\psi(z)-\langle v, \bar{z}-z\rangle\geq0.$
\end{proposition}
%
%


\section{Inexact proximal point method for nonconvex optimization}
This section contains two subsections. The first one states the composite nonconvex optimization (CNO) problem and 
discusses some notions of approximate solutions.
The second subsection proposes and analyzes a general framework for solving nonconvex optimization problems and
shows under very mild conditions that the composite gradient method is an instance of the general framework.


\subsection{The CNO problem and corresponding approximate solutions}\label{subsec:approxsol}
This subsection describes the CNO problem which will be the main subject of our analysis in
Subsection \ref{sec:AIPPmet}.
It also describes different notions of approximate solutions for the CNO problem and discusses their relationship. 

The CNO problem we are interested in is \eqref{eq:PenProb2Intro} where the following conditions are assumed to hold:
\begin{itemize}
\item[(A1)]$h \in \bConv{n}$;  
\item[(A2)] $g$ is a   differentiable function on $\dom h$ which, for some $M\geq m>0$, satisfies
\begin{equation}\label{ineq:concavity_g}
- \frac{m}{2}\|u-z\|^{2}\leq g(u)-\ell_g(u;z) \leq\frac{M}{2}\|u-z\|^{2}
\quad\forall z,u\in \dom h;
\end{equation}
\item[(A3)] $\phi_*>-\infty$.
\end{itemize}

We now make a few remarks about the above assumptions.  First, if $\nabla g$ is assumed to be $M$-Lipschitz
continuous, then \eqref{ineq:concavity_g} holds with $m=M$. However, our interest is in the case where
$0<m \ll M$ since this case naturally arises in the context of penalty methods for solving
linearly constrained composite nonconvex optimization problems as will be seen in Section \ref{sec:penaltymet}.
Second, it is well-known that a necessary condition for $z^*\in\dom  h$ to be a local minimum of \eqref{eq:PenProb2Intro} is that
$z^*$ be a stationary point of $g+h$, i.e.,  $0 \in \nabla g(z^*)+\partial h(z^*)$.

The latter inclusion motivates the following notion of approximate solution for problem \eqref{eq:PenProb2Intro}:
for a given tolerance $\hat \rho>0$,  a pair $( \hat{z},\hat{v})$ is called a $\hat\rho$-approximate solution of \eqref{eq:PenProb2Intro} if
\begin{equation} \label{eq:ref4'''}
\hat {v} \in \nabla g(\hat{z}) + \partial h (\hat{z}), \quad \|\hat{v}\| \le  \hat \rho.
\end{equation}
Another notion of approximate solution that naturally arises in our analysis of the general framework of  Subsection~\ref{subsec:GIPP} is as follows.
For a given  tolerance pair $(\bar \rho, \bar \varepsilon) \in \Re^2_{++}$,  a quintuple $(\lambda,z^-, z, w,\varepsilon) \in \Re_{++}\times\Re^n \times  \Re^n \times \Re^n \times \Re_+$ is called a $(\bar \rho, \bar \varepsilon)$-prox-approximate solution 
of \eqref{eq:PenProb2Intro} if
\begin{equation} \label{eq:ref4'}
w \in \partial_{ \varepsilon}\left( \phi + \frac{1}{2\lam} \|\cdot - z^- \|^2 \right)  (z),
\quad \left\|\frac{1}\lam (z^--z) + w \right\| \le  \bar \rho, \quad \varepsilon \le  \bar \varepsilon.
\end{equation}

Note that the first definition of approximate solution above depends on the composite structure $(g,h)$ of $\phi$ but the second one does not.

The next proposition, whose proof is presented in Appendix~\ref{app:approxsollemma}, shows how an approximate solution as in \eqref{eq:ref4'''}
can be obtained from a prox-approximate solution by performing a composite gradient step.

\begin{proposition}\label{prop:refapproxsol}
Let $h \in \bConv{n}$ 
and $g$ be a differentiable function on $\dom h$ whose gradient  satisfies the second inequality in \eqref{ineq:concavity_g}.
Let $(\bar \rho,\bar \varepsilon) \in \Re^2_{++}$ and a  $(\bar \rho,\bar \varepsilon)$-prox-approximate solution 
$(\lambda,z^-,z, w,\varepsilon)$ be given  
 and   define
\begin{align}\label{eq:def_zg} 
z_g &:= \argmin_u \left\{ \ell_g(u; z) + h(u) + \frac{M+\lambda^{-1}}2 \|u-z\|^2  \right \},\\
q_g&:=[M+\lambda^{-1}](z-z_g),\label{eq:def_qg}\\
\delta_g& := h(z) - h(z_g) - \inner{q_g - \nabla g(z) }{z-z_g}, \label{eq:def_deltag} \\
v_g &:=  q_g+\nabla g(z_g)  - \nabla g(z).\label{eq:def_vg} 
\end{align}
Then, the following statements hold:
\begin{itemize}
\item[(a)] $q_g \in \nabla g(z) + \partial  h(z_g)$ and
\[
\left( M+\lam^{-1} \right) \|z-z_g\| = \|q_g\|  \le \bar{\rho} + \sqrt{2\bar{\varepsilon} (M+\lam^{-1}) };
\]
\item[(b)] $\delta_g \ge 0$, $q_g \in \nabla g(z) + \partial_{\delta_g}  h(z)$ and
\[
\|q_g\|^2 + 2 (M+\lam^{-1}) \delta_g  \le  \left  [ \bar\rho + \sqrt{2\bar\varepsilon (M+\lam^{-1}) } \right ]^2;
\]
\item[(c)] if $\nabla g$ is
$M$-Lipschitz continuous on $\dom h$, then 
\[
v_g \in \nabla g(z_g) + \partial  h(z_g), \quad \|v_g\| \le 2 \|q_g\| \le 2 \left  [ \bar\rho + \sqrt{2\bar{\varepsilon} (M+\lam^{-1}) } \right ].
\]
\end{itemize}
\end{proposition}

Proposition \ref{prop:refapproxsol} shows that a prox-approximate solution yields three possible ways of measuring the quality of an approximate solution of \eqref{eq:PenProb2Intro}.
Note that the ones described in (a) and (b) do not assume $\nabla g$ to be Lipschitz continuous while the one in (c) does. This paper
only derives complexity results with respect to prox-approximate solutions and approximate solutions as in (c)
but we remark that complexity results for the ones in (a) or (b) can also be obtained.
Finally, we note that Lemma \ref{lem:approxsolff} in Appendix~\ref{app:approxsollemma} provides an alternative way of constructing approximate solutions as in (a), (b) or (c)
from a given prox-approximate solution.

%
%


\subsection{A general inexact proximal point framework}\label{subsec:GIPP}

This subsection introduces a general inexact  proximal point (GIPP) framework for solving the CNO problem \eqref{eq:PenProb2Intro}.

Although our main goal is to use the GIPP framework in the context of the CNO problem, we will describe it in
the context of the following more general problem
\begin{equation}\label{prob1_NC}
\phi_* := \inf \{ \phi(z) : z \in \Re^n \}
\end{equation}
where $\phi:\Re^{n}\to(-\infty,\infty]$ is a proper lower semi-continuous function, and $\phi_* > -\infty$. 

We now state the GIPP framework for computing prox-approximate solutions of \eqref{prob1_NC}.
\\
\noindent\rule[0.5ex]{1\columnwidth}{1pt}

GIPP Framework

\noindent\rule[0.5ex]{1\columnwidth}{1pt}

\begin{itemize}
\item [(0)] Let $\sigma \in (0,1)$ and $z_0 \in \dom  \phi $ be given, and set $k=1$; 
\item [(1)] find a quadruple $(\lambda_k,z_k,\tilde v_k,\tilde \varepsilon_k) \in \Re_{++}\times\Re^n \times \Re^n \times \Re_+$ satisfying
\begin{align}\label{inclusion:GIPPF}
& \tilde v_k \in \partial_{\tilde \varepsilon_k} \left ( \lambda_k\phi + \frac1{2} \| \cdot-z_{k-1} \|^2 \right)(z_k), \\
\label{eq:errocritGIPP}
& \|\tilde v_k\|^{2}+2\tilde \varepsilon_k\leq \sigma\|z_{k-1}-z_k+\tilde v_k\|^{2};
\end{align}
\item [(2)] set $k \leftarrow k+1$ and go to (1).
\end{itemize}
\rule[0.5ex]{1\columnwidth}{1pt}

%
%
%
%
%
%

Observe that GIPP framework is not a well-specified algorithm but rather a conceptual framework consisting of (possibly many) specific instances.
In particular, it does not specify how the quadruple $(\lambda_k,z_k,\tilde v_k,\tilde \varepsilon_k)$ is computed and whether it exists. These two issues
will depend on the specific instance under consideration and the properties assumed about problem \eqref{prob1_NC}.
In this paper, we will discuss two specific instances of the above GIPP framework for solving
\eqref{eq:PenProb2Intro}, namely, the composite gradient method briefly discussed at  the end
of this subsection and an accelerated proximal method presented in Subsection~\ref{sec:AIPPmet}.
In both of these instances, the sequences $\{\tv_k\}$ and $\{\tilde \varepsilon_k\}$ are non-trivial
(see Proposition~\ref{prop:gradient method} and Lemma~\ref{prop:AIPPmethod}(c)).

Let $\{(\lambda_k,z_k,\tilde v_k,\tilde \varepsilon_k)\}$ be the sequence generated by an instance of the GIPP framework
and consider the sequences $\{(r_k, v_k,\varepsilon_k)\}$ defined as
\begin{equation}\label{def:rveps}
(v_k,\varepsilon_k) := \frac{1}{\lambda_k} (\tilde v_k,\tilde \varepsilon_k), \quad
r_k := \frac{z_{k-1}-z_k}{\lambda_k}.
\end{equation}
Then, it follows from \eqref{inclusion:GIPPF} that the quintuple
$(\lambda,\hat z, z, v,\varepsilon)=(\lambda_k,z_{k-1},z_k,v_k,\varepsilon_k)$ satisfies the inclusion in \eqref{eq:ref4'} for every $k \ge 1$.
In what follows, we will derive the iteration complexity for the
quintuple $(\lambda_k,z_{k-1},z_k,v_k,\varepsilon_k)$ to satisfy: i) the first inequality in \eqref{eq:ref4'} only,
namely, $\|v_k+r_k\| \le \bar \rho$; and
ii) both inequalities in \eqref{eq:ref4'}, namely, $\|v_k+r_k\| \le \bar \rho$ and
$\varepsilon_k \le \bar \varepsilon$, and hence a $(\bar \rho,\bar \varepsilon)$-prox-approximate
solution of \eqref{eq:PenProb2Intro}.

Without necessarily assuming that the error condition \eqref{eq:errocritGIPP} holds, the following technical but straightforward result
derives bounds on
$\tilde \varepsilon_k$ and $\|\tilde v_k+z_{k-1}-z_k\|$
in terms of the  
quantities
\begin{equation}\label{def:rk&Lambak&deltak}
\delta_k=\delta_k(\sigma):= \frac{1}{\lambda_k}\max\left\{0,\|\tilde v_k\|^{2}+2\tilde \varepsilon_k-\sigma\|z_{k-1}-z_k+\tilde v_k\|^{2}\right\},\quad\Lambda_k:=\sum_{i=1}^k\lambda_i
\end{equation}
where $\sigma \in [0,1)$ is a given parameter. Note that if \eqref{eq:errocritGIPP} is assumed then $\delta_k=0$.
\begin{lemma}\label{LemGIPP1}
Assume that the sequence $\{(\lambda_k,z_k,\tilde v_k,\tilde \varepsilon_k)\}$ satisfies \eqref{inclusion:GIPPF}
and let $\sigma \in (0,1)$ be given. Then, for every $k\geq 1$, there holds
\begin{equation}\label{eq:pre-complex}
 \frac{1}{\sigma\lambda_k}\left(\|\tilde v_k\|^2+2\tilde{\varepsilon}_k-\lambda_k\delta_k\right)\leq \frac{1}{\lambda_k}\|z_{k-1}-z_k+\tilde v_k\|^{2}\leq \frac{2[\phi(z_{k-1})-\phi(z_k)]+\delta_k}{1-\sigma}
\end{equation}
where $\delta_k$ is as in \eqref{def:rk&Lambak&deltak}.   
\end{lemma}

\begin{proof}
First note that the  inclusion in \eqref{inclusion:GIPPF} is equivalent to 
\[
\lambda_i\phi(z)+\frac{1}{2}\|z-z_{i-1}\|^{2}\ge\lambda_i\phi(z_i)+\frac{1}{2}\|z_i-z_{i-1}\|^{2}+\left\langle \tilde v_i,z-z_i\right\rangle -\tilde \varepsilon_i\qquad\forall z\in\Re^{n}.
\]
Setting $z=z_{i-1}$ in the above inequality and using the definition
of $\delta_i$ given in \eqref{def:rk&Lambak&deltak}, we obtain 
\begin{align*}
&\lambda_i(\phi(z_{i-1})-\phi(z_i))  \ge\frac{1}{2}\left(\|z_{i-1}-z_i\|^{2}+2\left\langle \tilde v_i,z_{i-1}-z_i\right\rangle -2\tilde \varepsilon_i\right)\\
 & =\frac{1}{2}\left[\|z_{i-1}-z_i+\tilde v_i\|^{2}-\|\tilde v_i\|^{2}-2\tilde \varepsilon_i\right] \geq\frac{1}{2}\left[(1-\sigma)\|z_{i-1}-z_i+\tilde v_i\|^{2}-\lambda_i\delta_i\right] 
\end{align*}
and hence the proof  of  the second inequality in \eqref{eq:pre-complex} follows after simple rearrangements. 
The first inequality in \eqref{eq:pre-complex} follows immediately from \eqref{def:rk&Lambak&deltak}.
\end{proof}

\vgap
\begin{lemma}\label{lem:phikd0} Let  $\{(\lambda_k,z_k,\tilde v_k,\tilde \varepsilon_k)\}$ be generated by an instance of the GIPP framework.  
Then, for every $u\in \Re^n$, there holds
\[
  \phi(z_k)\leq \phi(u) + \frac{1}{2(1-\sigma) \lambda_k} \|z_{k-1}-u\|^2,\quad \forall k \geq 1.
\]
\end{lemma}
\begin{proof}
Using a simple algebraic manipulation, it is easy to see that \eqref{eq:errocritGIPP} yields
\begin{equation}\label{ineq:epstilde2}
\inner{\tilde v_k}{z_k- z_{k-1}} + \frac{1}\sigma \tilde \varepsilon_k - \frac12 \|z_{k-1}-z_k\|^2 \le -\frac{1-\sigma}{2\sigma}\| \tilde v_k\|^{2}.
\end{equation}
Now, letting $\theta := (1-\sigma)/\sigma>0$,  recalling definition \eqref{eq:epsubdiff}, using \eqref{inclusion:GIPPF} and \eqref{ineq:epstilde2},
and the fact that $\langle v,v'\rangle \leq (\theta/2)\|v\|^2+(1/2\theta)\|v'\|^2$ for all $v,v'\in \Re^n$,
we conclude that
\begin{align*}
 \lambda_k[\phi(z_k)-\phi(u)] &\le \frac{1}{2}\|z_{k-1}-u\|^{2}  +
 \left\langle \tilde v_k,z_k - u\right\rangle  + \tilde \varepsilon_k - \frac{1}{2}\|z_k-z_{k-1}\|^{2}  \\
& \leq \frac{1}{2}\|z_{k-1}-u\|^{2} + 
 \left\langle \tilde v_k,z_{k-1}-u  \right\rangle    -\frac{1-\sigma}{2\sigma}\| \tilde v_k\|^{2} \\
& \le \frac{1}{2}\|z_{k-1}-u\|^{2} +\left( \frac{\theta}2 \|\tilde v_k\|^2 + \frac{1}{2\theta} \|z_{k-1}-u\|^{2}   \right) -   \frac{1-\sigma}{2\sigma} \| \tilde v_k\|^{2}
\end{align*}
and hence that the conclusion of the lemma holds due to the definition of $\theta$.
\end{proof}

Let $z_0\in \Re^n, \sigma\in (0,1)$, and $ \lambda \ge 0$ be given and consider the following quantity 
\begin{equation}\label{eq:def_Rphi}
R(\phi;\lambda) := \inf \left \{ R(u;\phi,\lambda) :=   \frac12 \|z_0-u\|^2 + (1-\sigma) \lam [\phi(u) - \phi_*] :  u \in \Re^n\right\}
\end{equation}
where $\phi_*$ is as in \eqref{prob1_NC}. Clearly, $R(u;\phi,\lambda) \in \Re_+$ for all $u \in \dom h$ and  $R(\phi;\lambda) \in \Re_+$.

\begin{proposition}
\label{corGIPP1}
Let  $\{(\lambda_k,z_k,\tilde v_k,\tilde \varepsilon_k)\}$ be
generated by an instance of the GIPP framework. Then, the following statements hold:
\begin{itemize}
\item[(a)]  for every $k\geq 1$,
\begin{equation} \label{eqLemaux:rasc}
\frac{1-\sigma}{2\lambda_k}\left \|z_{k-1}-z_k+\tilde v_k\right\|^{2}\leq \phi(z_{k-1})-\phi(z_k);
\end{equation}
\item[(b)]
for every $k>1$, there exists $i\leq k$ such that
\begin{equation}\label{eq:corGIPP_complexaa}
\frac{1}{\lambda_i^2}\|z_{i-1}-z_i+\tilde v_i\|^2\leq \frac{2 R(\phi;\lambda_1) }{(1-\sigma)^2 \lam_1 (\Lambda_k-\lambda_1)}
\end{equation}
where  $\Lambda_k$ and $R(\cdot;\cdot)$ are   as in  \eqref{def:rk&Lambak&deltak} and \eqref{eq:def_Rphi}, respectively.
\end{itemize}
\end{proposition}
\begin{proof}  (a)
The proof of \eqref{eqLemaux:rasc}  follows immediately from \eqref{eq:pre-complex} and the fact that \eqref{eq:errocritGIPP} 
is equivalent to
$\delta_k=0$. 

%
(b) It follows from definitions of $\phi_*$ and $R(\cdot;\cdot,\cdot)$  in \eqref{prob1_NC} and \eqref{eq:def_Rphi}, respectively, \eqref{eqLemaux:rasc} and Lemma~\ref{lem:phikd0} with $k=1$
that for all $u\in \Re^n$,
\begin{align*}
&  \frac{R(u;\phi,\lambda_1)}{(1-\sigma)\lam_1}= \frac{1}{2(1-\sigma)\lambda_1} \|z_0-u\|^2+\phi(u)-\phi_*
\ge\phi(z_1)-\phi_* \geq  \sum_{i=2}^{k}[\phi(z_{i-1})-\phi(z_{i})] \\
& \geq  (1-\sigma) \sum_{i=2}^{k} \frac{ \left\|z_{i-1}-z_{i}+\tilde v_{i}\right\|^{2} }{2 \lam_i}
\geq \frac{(1-\sigma)(\Lambda_{k}-\lambda_1)}{2} \min_{i \le  k} \frac{1}{\lambda_i^2}\left \|z_{i-1}-z_{i}+\tilde v_{i}\right\|^{2}
\end{align*}
and hence that \eqref{eq:corGIPP_complexaa} holds in view of the definition of $R(\cdot;\cdot)$ in \eqref{eq:def_Rphi}.
\end{proof}

Proposition \ref{corGIPP1}(a) shows that GIPP enjoys the descent property \eqref{eqLemaux:rasc} which many frameworks and/or algorithms for solving \eqref{prob1_NC}
also share. It is worth noting that, under the assumption that $\phi$ is a KL-function, 
frameworks and/or algorithms sharing this property have been developed for example  in \cite{Attouch2009,Attouch2011,chouzenoux2016block,frankel2015splitting} where it is shown that
 the generated sequence $\{z_k\}$  converges to 
some stationary point of \eqref{prob1_NC} with a well-characterized asymptotic (but not global) convergence rate,
as long as $\{z_k\}$ has an accumulation point.



The following result, which follows immediately from Proposition \ref{corGIPP1}, considers the instances of the GIPP framework in which $\{\lambda_k \}$ is constant. For the purpose of stating it,
define
\begin{equation} \label{eq:dist0}
d_0:= \inf\{\|z_0-z^*\|:z^*\; \mbox{\rm is an optimal solution of \eqref{prob1_NC}} \}.
\end{equation}
Note that $d_0 < \infty$ if and only if \eqref{prob1_NC} has an optimal solution in which case
the above infimum can be replaced by a minimum in view of the first assumption following
\eqref{prob1_NC}.

\begin{corollary}\label{proprasc1}
Let $\{(\lam_k,z_k,\tilde v_k,\tilde \varepsilon_k)\}$ be generated by an instance the GIPP framework in which
$\lambda_k=\lambda$  for every $k \ge 1$, and define
$\{(v_k,\varepsilon_k,r_k)\}$ as in \eqref{def:rveps}. 
Then, the following statements hold:
\begin{itemize}
\item[(a)]  for every $k> 1$, there exists $i\leq k$ such that
\begin{equation}\label{eq:corGIPP_complex2}
\frac{1}{\lambda^2}\|z_{i-1}-z_i+\tilde v_i\|^2\leq \frac{2 R(\phi;\lambda)}{\lambda^2(1-\sigma)^2(k-1)}\leq\frac{\min\left\{2[\phi(z_0)-\phi_*], 
\frac{ d_0^2}{(1-\sigma)\lambda}\right\}}{\lambda(1-\sigma)(k-1)}
\end{equation}
where  $R(\cdot;\cdot)$ and $d_0$ are  as in   \eqref{eq:def_Rphi} and \eqref{eq:dist0}, respectively;
\item[(b)]  for any given tolerance $\bar \rho>0$, the GIPP generates a quintuple $(z^-,z,\tilde v,\tilde \varepsilon)$ such that 
$\|z^- -z+\tilde v\|\leq \lambda \bar \rho$ in a number of iterations bounded by
\begin{equation}\label{eq:complexCor23}
\left\lceil\frac{2 R(\phi;\lambda)}{\lambda^2(1-\sigma)^2 \bar \rho^2}+1\right\rceil.
\end{equation}
\end{itemize}
\end{corollary}
\begin{proof} 
(a) The proof of the first inequality follows immediately from Proposition~\ref{corGIPP1}(b)  and the fact that $\lambda_k=\lambda$ for every $k \ge 1$.
Now, note that due to \eqref{eq:def_Rphi}, we have $R(\phi;\lambda) \leq R(z_0;\phi,\lambda)=(1-\sigma)\lambda[\phi(z_0)-\phi_*$] and
$R(\phi;\lambda)\leq R(z^*;\phi,\lambda)=\|z_0-z^*\|^2/2$ for every optimal solution $z^*$  of \eqref{prob1_NC}.
The second inequality now follows from the previous observation and the definition of $d_0$ in \eqref{eq:dist0}. 

(b) This statement  follows immediately from the first inequality in (a).
\end{proof}


In the above analysis, we have assumed that $\phi$ is quite general.
On the other hand, the remaining part of this subsection assumes that $\phi$ has the composite structure as in \eqref{eq:PenProb2Intro},
i.e., $\phi=g+h$ where $g$ and $h$ satisfy conditions (A1)-(A3) of Subsection \ref{subsec:approxsol}.

We now briefly discuss some specific instances of the GIPP framework.
Recall that, for given  stepsize  $\lambda>0$ and initial point $z_0\in \dom  h$, 
the composite gradient method  for solving the CNO problem \eqref{eq:PenProb2Intro} computes 
recursively a sequence $\{z_k\} $ given by 
\begin{equation}\label{gradientmethod}
z_{k}=\argmin_z \left\{ \ell_g(z;z_{k-1}) +\frac{1}{2\lambda}\left\|z- z_{k-1}\right\|^2+h(z)\right\}
\end{equation}
where $\ell_g(\cdot;\cdot)$ is defined in \eqref{eq:defell}.
Note that if $h$ is the indicator  function of a closed convex set then the above  scheme reduces to the classical projected gradient method. 

The following result, whose proof is given in Appendix~\ref{app:proofgradmet}, shows that the composite gradient method with
$\lam$ sufficiently small is
a special case of the GIPP framework in which $\lambda_k = \lambda$ for all $k$.

\begin{proposition}\label{prop:gradient method} Let $\{z_k\}$ be generated by the composite gradient method \eqref{gradientmethod} with
$\lambda \leq 1/m$ and $\lambda < 2/M$, and define $\tilde v_k:=z_{k-1}-z_k$, $\lambda_k:=\lambda$ and
\begin{equation}\label{eq:statCGM}
  \tilde \varepsilon_k:=\lambda\left[g(z_k)-\ell_g(z_k;z_{k-1}) +\frac{1}{2\lambda}\|z_k-z_{k-1}\|^2\right].
\end{equation}
Then, the quadruple $(\lambda_k,z_k,\tilde v_k,\tilde \varepsilon_k)$ satisfies the inclusion \eqref{inclusion:GIPPF} with $\phi=g+h$, and the relative error condition \eqref{eq:errocritGIPP} with $\sigma := (\lambda M+2)/4$.
Thus, the composite gradient method \eqref{gradientmethod} can be seen as an instance of the GIPP framework.
\end{proposition}

Under the assumption that $\lambda<2/M$ and $\nabla g$ is  $M$-Lipschitz continuous, it is well-known that
the composite gradient method obtains a $\hat \rho$-approximate solution in
$\mathcal{O}([\phi(z_0)-\phi_*]/(\lambda\hat \rho^2))$ iterations.
On the other hand,  under the assumption that $\lambda \le 1//M$ and $\nabla g$ is  $M$-Lipschitz continuous, we can easily see
that the above result together with  Corollary~\ref{proprasc1}(b) imply that the composite gradient method obtains a $\hat \rho$-approximate solution in
$\mathcal{O}(R(\phi;\lambda)/(\lambda^2\hat \rho^2))$ iterations.

We now make a few general remarks about our discussion in this subsection so far. 
First, the condition on the stepsize $\lambda$ of Proposition~\ref{prop:gradient method} forces it to be ${\cal O}(1/M)$ and hence quite small whenever
$M\gg m$.
Second, Corollary~\ref{proprasc1}(b) implies that the larger $\lambda$ is, the smaller the complexity bound \eqref{eq:complexCor23} becomes.
Third,
letting $\lambda_k=\lam$  in the GIPP framework for some $\lambda \le 1/m$ guarantees that the
function $\lambda_k \phi  + \|\cdot-z_{k-1}\|^2/2$ which appears in \eqref{inclusion:GIPPF} is convex.

In the remaining part of this subsection, we briefly outline the ideas behind an accelerated instance of the GIPP framework
which chooses $\lambda=\mathcal{O}(1/m)$.
 First, note that when $\sigma=0$, \eqref{inclusion:GIPPF} and \eqref{eq:errocritGIPP} imply that $(\tv_k,\tilde \varepsilon_k)=(0,0)$ and
 \begin{equation}\label{inclusion:GIPPF'}
0 \in \partial \left ( \lambda_k\phi + \frac1{2} \| \cdot-z_{k-1} \|^2 \right)(z_k).
\end{equation}
and hence that $z_k$ is an optimal solution of the prox-subproblem
\begin{equation}\label{eq:prox-sub}
z_{k}=\argmin_z \left\{ \lambda_k \phi (z)+\frac{1}{2}\left\|z- z_{k-1}\right\|^2\right\}.
\end{equation}
More generally, assuming that \eqref{eq:errocritGIPP} holds for some $\sigma > 0$ gives us an interpretation of
$z_k$, together with $(\tilde v_k,\tilde \varepsilon_k)$,
as being an approximate solution of  \eqref{eq:prox-sub} where its (relative) accuracy is measured by the $\sigma$-criterion
\eqref{eq:errocritGIPP}.
Obtaining such an approximate solution is  generally  difficult unless the objective function of the prox-subproblem \eqref{eq:prox-sub} is convex. This suggests choosing $\lambda_k=\lambda$ for some $\lam \le 1/m$ which, according to a remark in the
previous paragraph, ensures that $\lambda_k \phi + (1/2)\|\cdot\|^2 $ is convex for every $k$,
and then applying an ACG method to the (convex) 
prox-subproblem~\eqref{eq:prox-sub} to obtain $z_k$ and a certificate pair $(\tilde v_k, \tilde \varepsilon_k)$ satisfying \eqref{eq:errocritGIPP}.
An accelerated prox-instance of the GIPP framework obtained in this manner will be the subject of Subsection \ref{sec:AIPPmet}.


\section{Accelerated gradient methods}\label{sec:accgradmet}
The main goal of this section is to present another instance of the GIPP framework where the triples $(z_k,\tilde v_k,\tilde \varepsilon_k)$ are obtained
by applying an ACG method to  the subproblem~\eqref{eq:prox-sub}.
It contains two subsections. The first one reviews an  ACG variant for solving a composite strongly convex optimization problem
and discusses some well-known and new results for it which will be useful in the analysis of the accelerated GIPP  instance.
Subsection~\ref{sec:AIPPmet} presents the accelerated  GIPP  instance for solving \eqref{eq:PenProb2Intro} and derives its corresponding  iteration-complexity bound.


\subsection{Accelerated gradient method for strongly convex optimization}\label{subsec:Nesterov's-Method}

This subsection reviews an ACG  variant and its convergence properties for solving the following optimization problem
\begin{equation}\label{mainprob:nesterov}
\min \{\psi(x):=\psi_s(x)+\psi_n(x) :  x \in \Re^n\}
\end{equation}
where the following conditions are assumed to hold
\begin{itemize}
\item [(B1)]$\psi_n:\Re^n\rightarrow (-\infty,+\infty]$ is a proper, closed and $\mu$-strongly convex  function with $\mu \ge 0$;
\item [(B2)]$\psi_s$ is a convex differentiable function on $\dom \psi_n$ which,  for some $L>0$, satisfies
$
\psi_s(u)-\ell_{\psi_s}(u; x)
\leq  L\|u-x\|^2/2$ for every $ x, u \in \dom \psi_n$
where $\ell_{\psi_s}(\cdot;\cdot)$ is defined in \eqref{eq:defell}.
\end{itemize}

The ACG variant (\cite{Attouch2016,YHe2,nesterov2012gradient,nesterov1983,tseng2008accmet}) for solving \eqref{mainprob:nesterov} is as follows.

\noindent\rule[0.5ex]{1\columnwidth}{1pt}

\textbf{ACG Method}

\noindent\rule[0.5ex]{1\columnwidth}{1pt}
\begin{itemize}
\item[(0)]  Let a pair of functions $(\psi_s,\psi_n)$  as in \eqref{mainprob:nesterov} 
and initial point $x_{0}\in \dom  \psi_n $ be given, and set $y_{0}=x_{0}$, $A_{0}=0$, $\Gamma_0\equiv0$ and $j=0$; 
\item[(1)]  compute
\begin{align*}
 A_{j+1}  &=A_{j}+\frac{\mu A_{j}+1+\sqrt{(\mu A_{j}+1)^2+4L(\mu A_{j}+1)A_{j}}}{2L},\\
\tilde{x}_{j}  &=\frac{A_{j}}{A_{j+1}}x_{j}+\frac{A_{j+1}-A_{j}}{A_{j+1}}y_{j},\quad\Gamma_{j+1}=\frac{A_{j}}{A_{j+1}}\Gamma_j+\frac{A_{j+1}-A_{j}}{A_{j+1}}\ell_{\psi_s}(\cdot;\tilde x_j),\\
y_{j+1} &=\argmin_{y}\left\{ \Gamma_{j+1}(y)+\psi_n(y)+\frac{1}{2A_{j+1}}\|y-y_{0}\|^{2}\right\},\\
x_{j+1} & =\frac{A_{j}}{A_{j+1}}x_{j}+\frac{A_{j+1}-A_{j}}{A_{j+1}}y_{j+1};
\end{align*}
\item[(2)] compute 
\begin{align*}
u_{j+1}&=\frac{y_0-y_{j+1}}{A_{j+1}},\\[2mm]
\eta_{j+1}&= \psi(x_{j+1})-\Gamma_{j+1}(y_{j+1})- \psi_n(y_{j+1})-\langle u_{j+1},x_{j+1}-y_{j+1}\rangle;   
\end{align*}
\item[(3)]  set $j\leftarrow j+1$ and go to (1). 
\end{itemize}
\noindent\rule[0.5ex]{1\columnwidth}{1pt}

Some remarks about the ACG method follow. First, the main core and usually the common way of describing an iteration of the ACG method is as in step~1.
Second, the
extra sequences $\{u_j\}$ and $\{\eta_j\}$ computed  in step~2 will be used to develop a stopping criterion for the ACG method when the latter
is called as a subroutine in the context of
the AIPP method stated in Subsection~\ref{sec:AIPPmet}.
Third, the ACG method in which  $\mu=0$  is a special case of a slightly more general one studied by Tseng in \cite{tseng2008accmet} (see Algorithm~3
of \cite{tseng2008accmet}).
The analysis of the general case of the ACG method in which $\mu\geq0$ was studied  in  \cite[Proposition~2.3]{YHe2}.

The next proposition summarizes the basic properties of the ACG method.

\begin{proposition}\label{prop1:nesterov} Let $\{(A_j,\Gamma_j,x_j,u_j,\eta_j)\}$     be the sequence generated by the ACG method applied to \eqref{mainprob:nesterov}
where $(\psi_s,\psi_n)$ is a given pair of data functions satisfying (B1) and (B2) with $\mu \geq 0$. Then,  the following statements hold
\begin{itemize}
\item[(a)] for every $j\ge1 $, we have $\Gamma_j\leq \psi_s$ and 
\begin{align}
\psi(x_j)  &\leq\min_{x}\left\{\Gamma_{j}(x)+\psi_n(x)+\frac{1}{2A_{j}}\|x-x_{0}\|^{2}\right\}, \label{ineq:psixj}\\[2mm]
A_{j}&\geq\frac{1}{L}\max\left\{\frac{j^{2}}{4},\left(1+\sqrt{\frac{\mu}{4L}}\right)^{2(j-1)}\right\};\label{ineq:increasingA_k}
\end{align}
\item[(b)] for every solution $x^*$  of \eqref{mainprob:nesterov}, we have 
\begin{equation}\label{eq:Nest_funct}
\psi(x_{j})-\psi(x^*)\leq\frac{1}{2A_{j}}\|x^*-x_{0}\|^{2} \qquad \forall j\geq 1;
\end{equation}
\item[(c)] for every $j\geq 1$, we have
\begin{equation}
u_j\in  \partial_{\eta_{j}}(\psi_s+\psi_n)(x_j) ,\quad\|A_{j}u_{j}+x_{j}-x_{0}\|^{2}+2A_{j}\eta_{j}\le\|x_{j}-x_{0}\|^{2}.\label{ineq:NestHPE}
\end{equation}
\end{itemize}
\end{proposition}
\begin{proof}
For the proofs of (a) and (b) see  \cite[Proposition 2.3]{YHe2}. 

(c) It follows from the optimality condition  for $y_{j}$  and  the definition of $u_j$ that
$u_j\in \partial (\Gamma_j +\psi_n) (y_j)$, for all $j\geq 1$. Hence  Proposition~\ref{prop:transpForm}  yields  
\begin{equation}\label{eq:subdifaux}
(\Gamma_j +\psi_n)(x)\geq (\Gamma_j +\psi_n) (x_j)+\langle u_j, x-x_j\rangle -\tilde \eta_j,\quad \forall x\in \Re^n,
\end{equation} 
where $\tilde \eta_j= (\Gamma_j+\psi_n)(x_j)-(\Gamma_j+\psi_n)(y_j)-\langle u_j, x_j-y_j\rangle\geq0$. Thus the inclusion in \eqref{ineq:NestHPE} follows from \eqref{eq:subdifaux}, the first statement in (a), and the fact that $\eta_j=\tilde \eta_j+\psi_s(x_j)-\Gamma_j(x_j)$.
Now, in order to prove the inequality in \eqref{ineq:NestHPE}, note that  $y_0=x_0$ and that the definitions of $u_j$ and $\eta_j$  yield 
\begin{align}
&\|A_{j}u_{j}+x_{j}-x_{0}\|^{2}-\|x_{j}-x_{0}\|^{2}= \|y_j-y_0\|^2+2\langle y_0-y_j, x_j-y_0\rangle\\[2mm]
&2A_j\eta_j=2A_j[\psi(x_j)-(\Gamma_j +\psi_n)(y_j)]+2\langle y_0-y_j,y_j-x_j \rangle.
\end{align}
Then adding the above two identities we obtain
\begin{align*}
\|A_{j}u_{j}+x_{j}-x_{0}\|^{2}&+2A\eta_j-\|x_{j}-x_{0}\|^{2}= 2A_j[\psi(x_j)-(\Gamma_j+\psi_n)(y_j)]-\|y_j-y_0\|^2\\[2mm]
&\leq 2A_j\left[\psi(x_j)- \left((\Gamma_j+\psi_n)(y_j)+\frac{1}{2A_j}\|y_j-y_0\|^2\right)\right].
\end{align*}
Hence, the inequality in \eqref{ineq:NestHPE} follows from the last inequality,  the definition of $y_j$ and  \eqref{ineq:psixj}.
\end{proof}
\vgap

The main role of the ACG variant of this subsection is to find an approximate solution $z_k$ of
the subproblem \eqref{inclusion:GIPPF} together with a certificate pair $(\tv_k,\tilde \varepsilon_k)$ satisfying
\eqref{inclusion:GIPPF} and \eqref{eq:errocritGIPP}. Indeed,  since   \eqref{eq:prox-sub} with $\lambda$ sufficiently small is a special case of \eqref{mainprob:nesterov}, we can apply the ACG method with $x_0=z_{k-1}$ to obtain the triple $(z_k,\tv_k,\tilde \varepsilon_k)$ satisfying \eqref{inclusion:GIPPF} and \eqref{eq:errocritGIPP}.

 The following result essentially analyzes the iteration-complexity to compute
the aforementioned triple.

\begin{lemma}\label{lem:nest_complex} 
Let $\{(A_j,x_j, u_j,\eta_j)\}$ be the sequence generated by the ACG method applied to \eqref{mainprob:nesterov}
where $(\psi_s,\psi_n)$ is a given pair of data functions satisfying (B1) and (B2) with $\mu\geq 0$.
Then, for any $\sigma>0$ and index $j$ such that $A_j\geq2(1+\sqrt{\sigma})^{2}/\sigma$, we have

\begin{equation}\label{ineq:Nest_vksigma}
\|u_j\|^{2}+2\eta_j\le\sigma\|x_{0}-x_j+u_j\|^{2}.
\end{equation}
As a consequence, the ACG method obtains  a triple $(x,u,\eta)=(x_j,u_j,\eta_j)$ satisfying
\[
u\in  \partial_{\eta}(\psi_s+\psi_n)(x) \quad \|u\|^{2}+2\eta\le\sigma\|x_{0}-x+u\|^{2}
\]
in at most $\left\lceil2\sqrt{2L}(1+\sqrt{\sigma})/\sqrt{\sigma}\right\rceil$ iterations.
\end{lemma}
\begin{proof} Using the triangle inequality for norms,   the relation $(a+b)^2\leq 2(a^2+b^2)$ for all $a, b \in \Re$, and  the inequality in \eqref{ineq:NestHPE}, we obtain 
 \begin{align*}
 \|u_j\|^2+2\eta_j 
 &\leq \max\{1/A_j^2,1/(2A_j)\}(\|A_ju_j\|^{2}+4A_j\eta_j)\\
 &\leq  \max\{1/A_j^2,1/(2A_j)\}(2\|A_ju_j+x_j-x_0\|^{2}+2\|x_j-x_0\|^{2}+4A_j\eta_j)\\
 &\leq  \max\{\left(2/A_j\right)^2,2/A_j\}\|x_j-x_0\|^{2} \leq   \frac{\sigma}{(1+\sqrt{\sigma})^2}\|x_j-x_0\|^{2}
\end{align*}
where the last inequality is due to $A_j\geq2(1+\sqrt{\sigma})^{2}/\sigma$.
On the other hand,  the triangle inequality for norms and  simple calculations  yield 
\begin{equation*}
\|x_j-x_0\|^{2}\leq(1+\sqrt{\sigma})\|x_0-x_j+u_j\|^2+\left(1+\frac{1}{\sqrt{\sigma}}\right)\|u_j\|^2.
\end{equation*} 
Combining the previous  estimates, we obtain
 \begin{equation}
 \|u_j\|^2+2\eta_j \leq \frac{\sigma}{1+\sqrt{\sigma}}\|x_0-x_j+u_j\|^2 + \frac{\sqrt{\sigma}}{1+\sqrt{\sigma}}\|u_j\|^2
 \end{equation}
 which easily implies   \eqref{ineq:Nest_vksigma}.  
Now if $j\geq \left\lceil2\sqrt{2L}(1+\sqrt{\sigma})/\sqrt{\sigma}\right\rceil
$  then it follows from \eqref{ineq:increasingA_k} that
$A_j\geq2(1+\sqrt{\sigma})^2/\sigma$ and hence, due to the first statement of the lemma,   \eqref{ineq:Nest_vksigma} holds. The last conclusion combined with the inclusion in \eqref{ineq:NestHPE} prove the last statement of the lemma.
\end{proof}

Note that Proposition~\ref{prop1:nesterov} and Lemma~\ref{lem:nest_complex} hold for any $\mu\ge 0$.
On the other hand, the next two results hold only for $\mu>0$ and 
derive some important relations satisfied by two distinct iterates of the ACG method.
They will be used later on in Subsection \ref{sec:AIPPmet} to analyze the refinement phase (step 3) of the 
AIPP method stated there.

\begin{lemma}\label{lem:Nest_prefut}
Let $\{(A_j,x_j,u_j,\eta_j)\}$ be generated by the  ACG method applied to \eqref{mainprob:nesterov}
where $(\psi_s,\psi_n)$ is a given pair of data functions satisfying (B1) and (B2) with $\mu>0$.  Then,
 \begin{equation}\label{ineq:Nestmu}
 \left(1-\frac{1}{\sqrt{A_j\mu}}\right)\|x^*-x_{0}\| \leq \|x_j-x_{0}\| \leq \left(1+\frac{1}{\sqrt{A_j\mu}}\right)\|x^*-x_{0}\|\qquad \forall j\geq 1,
\end{equation}
where $x^*$ is the unique solution of  \eqref{mainprob:nesterov}.
As a consequence, for all indices $i, j \ge 1$ such that $A_i\mu >1$, we have
 \begin{equation}\label{eq:prefut}
 \|x_j-x_0\| \leq 
 \left(\frac{1+\frac{1}{\sqrt{A_j\mu}}}{1-\frac{1}{\sqrt{A_i\mu}}}\right)\|x_i-x_0\|.
 \end{equation}
\end{lemma}
\begin{proof}
First note that condition (B1) combined with \eqref{mainprob:nesterov} imply that  $\psi$ is $\mu$-strongly convex. Hence,  it follows from \eqref{eq:Nest_funct} that 
$$
\frac{\mu}{2}\|x_j-x^*\|^2\leq \psi(x_{j})-\psi(x^*) \le  \frac{1}{2A_j}\|x^*-x_0\|^2
$$
and hence that
\begin{equation}\label{eq:auxnest00}
\|x_j-x^*\|\leq \frac{1}{\sqrt{ A_j\mu}}\|x^*-x_0\|.
\end{equation}
The inequalities
\[
\|x^*-x_0\|-\|x_j-x^*\|\leq\|x_j-x_0\|\leq  \|x_j-x^*\|+\|x^*-x_0\|,
\]
 which are due to the triangle inequality for norms,
together  with \eqref{eq:auxnest00} clearly implies \eqref{ineq:Nestmu}. The last statement of the lemma follows immediately from \eqref{ineq:Nestmu}. 
\end{proof}

As a consequence of Lemma \ref{lem:Nest_prefut}, the following result obtains several important relations
on certain quantities corresponding to two arbitrary iterates of
the ACG  method.

\begin{lemma}\label{prop_Nestxjxi}Let $\{(A_j,x_j,u_j,\eta_j)\}$ be generated by the ACG method applied to \eqref{mainprob:nesterov}
where $(\psi_s,\psi_n)$ is a given pair of data functions satisfying (B1) and (B2) with $\mu> 0$.  Let $i$ be an index such that $A_i\geq\max\{8,9/\mu\}$. Then,  for every $j\geq i$, we have
 \begin{equation}\label{eq:lemNestaaa}
 \|x_j-x_0\| \leq  2\|x_i-x_0\|,\quad \|u_j\|\leq \frac{4}{A_j}\|x_i-x_0\|, \quad
\eta_j\leq \frac{2}{A_j}\|x_i-x_0\|^2,
 \end{equation}
\begin{equation}\label{eq:corNest01}
\|x_0-x_j + u_j\|\leq \left(4+\frac{8}{A_j}\right)\|x_0-x_i+u_i\|,\qquad \eta_j\leq \frac{8\|x_0-x_i+u_i\|^2}{A_j}.
\end{equation}
\end{lemma}
\begin{proof}
 The first inequality in \eqref{eq:lemNestaaa} follows from \eqref{eq:prefut} and the assumption that
  $A_i\mu\geq 9$. Now, using the inequality in \eqref{ineq:NestHPE} and the triangle inequality for norms, we easily see that
 $$
 \|u_j\|\leq\frac{2}{A_j}\|x_j-x_0\|,\quad \eta_j\leq\frac{1}{2A_j}\|x_j-x_0\|^2
 $$
 which, combined with the first inequality in  \eqref{eq:lemNestaaa},
 prove the second and the third inequalities in \eqref{eq:lemNestaaa}. 
Noting that $A_i\geq 8$ by assumption,
 Lemma \ref{lem:nest_complex} implies that \eqref{ineq:Nest_vksigma} holds with $\sigma=1$ and $j=i$, and hence that
  \begin{equation}\label{eq:lemNestbbb}
  \|u_i\|\leq \|x_0-x_i+ u_i\|.
  \end{equation}  
Using the triangle inequality, the first two inequalities in \eqref{eq:lemNestaaa} and relation \eqref{eq:lemNestbbb}, we
conclude that
\begin{align*}
\|x_0-x_j + u_j\|&\leq \|x_0-x_j\|+\|u_j\|\leq \left(2+\frac{4}{A_j}\right)\|x_0-x_i\|\nonumber\\
&\leq \left(2+\frac{4}{A_j}\right)\left(\|x_0-x_i+ u_i\|+\|u_i\|\right)\leq\left(4+\frac{8}{A_j}\right)\|x_0-x_i+ u_i\|,
\end{align*}
and that the first inequality in \eqref{eq:corNest01} holds.
Now, the last inequality in \eqref{eq:lemNestaaa}, combined with the triangle inequality for norms and
the relation $(a+b)^2\leq 2(a^2+b^2)$, imply that
$$
\eta_j\leq \frac{2}{A_j}\|x_0-x_i\|^2\leq \frac{4}{A_j}\left(\|x_0-x_i+u_i\|^2+\|u_i\|^2\right).
$$
Hence,  in view of \eqref{eq:lemNestbbb},  the last inequality in \eqref{eq:corNest01} follows.
\end{proof}


\subsection{The AIPP method}\label{sec:AIPPmet}
This subsection introduces and analyzes the AIPP method to compute approximate solutions of the CNO problem \eqref{eq:PenProb2Intro}.
The main results of this subsection are Theorem~\ref{th:AIPPcomplexity} and Corollary~\ref{cor:AIPPref2}
which analyze the iteration-complexity of the AIPP method to obtain  approximate solutions of the CNO problem in the sense 
of  \eqref{eq:ref4'} and \eqref{eq:ref4'''},  respectively.

We  start by stating  the AIPP method.

\noindent\rule[0.5ex]{1\columnwidth}{1pt}

AIPP Method

\noindent\rule[0.5ex]{1\columnwidth}{1pt}
\begin{itemize}
\item [(0)] Let $z_0 \in \dom h $,  $\sigma \in (0,1)$, a pair $(m,M)$ satisfying  \eqref{ineq:concavity_g}, a scalar  $0 < \lambda\leq 1/(2m)$ and a  tolerance pair  $(\bar\rho,\bar\varepsilon) \in \Re^2_{++}$ be given,  and set  $k=1$; 
\item [(1)] perform at least $\left\lceil 6\sqrt{2\lambda M+1}\right\rceil$ iterations of the ACG method
started from $z_{k-1}$
and with 
\begin{equation}\label{psi_AIPP}
\psi_s=\psi_s^k:=\lambda g+\frac{1}{4}\|\cdot-z_{k-1}\|^2, \quad \psi_n=\psi_n^k:=\lam h +\frac{1}{4}\|\cdot-z_{k-1}\|^2
\end{equation}
to obtain a triple  $(x,u,\eta) \in \Re^n \times \Re^n \times \Re_+$ satisfying
\begin{equation}\label{eq:AIPPM_hpe}
u \in \partial_{\eta} \left(\lambda (g+h)+\frac{1}{2}\|\cdot-z_{k-1}\|^2\right)(x),
\quad \|u\|^2+2\eta \leq \sigma\|z_{k-1}-x+ u \|^2;
\end{equation}
\item [(2)] if
\begin{equation}\label{ineq:stopAIPPM}
\|z_{k-1}-x+u\| \leq \frac{\lambda \bar \rho}{5},
\end{equation}
then go to (3); otherwise  set $(z_k,\tilde v_k,\tilde \varepsilon_k)=(x,u,\eta)$, $k \leftarrow k+1$ and go to (1);
\item[(3)]
restart the previous call to the ACG method in step 1 to find an iterate
$(\tx,\tilde u,\tilde \eta)$ satisfying \eqref{eq:AIPPM_hpe} with $(x,u,\eta)$ replaced by $(\tx,\tilde u,\tilde \eta)$ and the extra condition
\begin{equation} \label{eq:newcri1}
\tilde \eta/\lam  \le \bar\varepsilon
\end{equation}
and set $(z_k,\tilde v_k,\tilde \varepsilon_k)=(\tx,\tilde u,\tilde \eta)$; finally,
output $(\lambda,z^-,z,w,\varepsilon)$ where $$(z^-,z,w,\varepsilon)=(z_{k-1}, z_k,\tv_k/\lambda, \tilde \varepsilon_k /\lambda).$$
\end{itemize}
\rule[0.5ex]{1\columnwidth}{1pt}


Some comments about the AIPP method are in order.  First, 
the ACG iterations performed in steps  1 and 3 are referred to as the inner iterations of the AIPP method.
Second, in view of the last statement of  Lemma~\ref{lem:nest_complex} with $(\psi_s,\psi_n)$ given by \eqref{psi_AIPP}, the ACG method obtains a triple
$(x,u,\eta)$ satisfying
 \eqref{eq:AIPPM_hpe}.
Observe that Proposition \ref{prop1:nesterov}(c) implies that every triple $(x,u,\eta)$  generated by the ACG method
satisfies the inclusion in \eqref{eq:AIPPM_hpe} and hence only the inequality in  \eqref{eq:AIPPM_hpe} needs to be checked for termination.
Third, the consecutive loops consisting of steps 1 and 2 (or, steps 1, 2 and 3 in the last loop) are referred to as the outer iterations of the AIPP method.
In view of \eqref{eq:AIPPM_hpe}, they can be viewed as iterations of the GIPP framework 
applied to the CNO problem \eqref{eq:PenProb2Intro}. 
Fourth, instead of running the ACG method by at least the constant number of iterations described
in step 1, one could run the more practical variant which stops  (usually, much earlier) whenever the second inequality in \eqref{eq:AIPPM_hpe} is satisfied.
We omit the tedious analysis and more complicated description of this AIPP variant, but remark that its iteration complexity is the same as
the one studied in this subsection.
Finally, the last loop supplements steps 1 and 2 with step 3 whose
goal is to obtain a triple $(\tx,\tilde u,\tilde \eta)$ with a possibly smaller $\tilde \eta$ while
preserving the quality of the quantity $\|z_{k-1}-\tilde x+\tilde u\|$ which
at its start is bounded by $\lambda\bar\rho/5$ and, throughout its inner iterations, can be shown to be bounded by $\lambda\bar\rho$
(see the first inequality in \eqref{eq:lemAIPPestimates}).


The next proposition summarizes some facts about the  AIPP method.
\begin{lemma}\label{prop:AIPPmethod}
The following statements about the AIPP method hold:
\begin{itemize}
\item[(a)]
at every outer iteration, the call to the ACG method in step 1 finds a triple
$(x,u,\eta)$ satisfying \eqref{eq:AIPPM_hpe} in at most
\begin{equation}\label{eq:AIPPcomplexinner}
\left\lceil \max\left\{\frac{2(1+\sqrt{\sigma})}{\sqrt{\sigma}},6\right\}\sqrt{2\lam M+1}\right\rceil
\end{equation}
inner iterations;
\item[(b)]
at the last outer  iteration, the extra number of ACG iterations to obtain the triple $(\tx,\tv,\tilde \eta)$ is bounded by
\begin{equation}\label{eq:lemAIPPcomplex3}
\left\lceil   2\sqrt{2\lambda M+1} \,  \log^+_1\left(\frac{2\bar\rho\sqrt{(2\lambda M+1)\lambda}}{5\sqrt{\bar\varepsilon}}\right) +1 \right\rceil;
\end{equation}
\item[(c)]
the AIPP method is a special implementation of the GIPP framework in which $\lambda_k=\lambda$  for every $k \ge 1$;  
\item[(d)]
the number of outer iterations performed by the AIPP method is bounded by 
\begin{equation}\label{def:T_p}
 \left\lceil\frac{25R(\phi;\lambda)}{(1-\sigma)^2\lambda^2{\bar\rho}^2}+1\right\rceil
\end{equation}
where $R(\cdot;\cdot)$ is as defined in \eqref{eq:def_Rphi}.
\end{itemize}
\end{lemma}
\begin{proof} 
(a)  First note that the function $\psi_s$ defined in \eqref{psi_AIPP} satisfies condition (B2) of Subsection~\ref{subsec:Nesterov's-Method}
with  $L= \lambda M+1/2$,  in view of  \eqref{ineq:concavity_g}. Hence, it follows from the last statement of Lemma~\ref{lem:nest_complex} that the ACG method  obtains a triple $(x,u,\eta)$ satisfying
\eqref{eq:AIPPM_hpe} in at most  
 \begin{equation}\label{eq:thAIPPa}
\left\lceil\frac{2\sqrt{2}(1+\sqrt{\sigma})}{\sqrt{\sigma}}\sqrt{\lambda M +1/2}\right\rceil
\end{equation}
inner iterations.
Hence, (a) follows from the above conclusion and the fact that the ACG method performs at least  $\left\lceil 6\sqrt{2\lambda M+1)}\right\rceil$ inner iterations, in view of  step~1.

(b)   Consider the triple $(x,u,\eta)$ obtained in step 1 during the last outer iteration of 
the AIPP method.
In view of step 1, there exists an index $i \ge \left\lceil 6\sqrt{2\lambda M+1}\right\rceil$ such that
$(x,u,\eta)$ is the $i$-iterate of the ACG method started from $x_0=z_{k-1}$ applied to problem \eqref{mainprob:nesterov} with $\psi_s$ and $\psi_n$ as in \eqref{psi_AIPP}. Noting that  the functions  $\psi_n$ and $\psi_s$ satisfy conditions (B1) and  (B2)
of Subsection~\ref{subsec:Nesterov's-Method} with  $\mu=1/2$ and $L= \lambda M+1/2$ (see \eqref{ineq:concavity_g}) and using
the above inequality on the index $i$ and relation \eqref{ineq:increasingA_k},
we conclude that $A_i\geq 18=\max\{8,9/\mu\}$, and hence that $i$ satisfies  the assumption of Lemma~\ref{prop_Nestxjxi}. It then follows from 
\eqref{eq:corNest01},  \eqref{ineq:stopAIPPM}  and \eqref{ineq:increasingA_k} that the continuation of the ACG method as in step~3 generates a  triple  $(\tilde x,\tilde u,\tilde \eta)=(x_j,u_j,\eta_j)$ satisfying 
\begin{equation}\label{eq:lemAIPPestimates}
\|z_{k-1}-\tilde x + \tilde u\|\leq \left(4+\frac{8}{A_j}\right)\frac{\lambda \bar\rho}{5}\leq \lambda \bar\rho,\qquad 
 \tilde \eta\leq \frac{8\lambda^2 \bar{\rho}^2}{25A_j}\leq \frac{8L\lambda^2\bar\rho^2}{25\left(1+\sqrt{\frac{\mu}{4L}}\right)^{2(j-1)}}.
\end{equation}
Noting the stopping criterion~\eqref{eq:newcri1} and using the last inequality above, the fact that $\mu=1/2$ and  $L= \lambda M+1/2$, and the relation
that $\log(1+t)\geq t/2$ for all $t \in [0,1]$, we can easily see that (b) holds. 

(c) This statement is obvious.

(d) This statement follows by combining (c),  the stopping criterion \eqref{ineq:stopAIPPM}, and Corollary~\ref{proprasc1}(b) with $\bar{\rho}$ replaced by $\bar{\rho}/5$.
\end{proof}

\vgap
Next we state one of our main results of this paper which derives the iteration-complexity of the AIPP method to obtain prox-approximate solutions of  the CNO problem in the sense of~\eqref{eq:ref4'}. 
Recall that the AIPP method assumes that $\lambda \le 1/(2m)$.

\begin{theorem}\label{th:AIPPcomplexity}
Under assumptions (A1)-(A3),
the AIPP method terminates with a prox-solution  $(\lambda,z^-,z,w,\varepsilon)$ 
 within
\begin{equation}\label{eq:lastboundAIPP}
\mathcal{O}\left\{ \sqrt{\lam M+1}\left[ \frac{R(\phi;\lambda)}{\lam^2\bar{\rho}^2} + \log^+_1\left(\frac{\bar\rho \sqrt{(\lambda M+1)\lambda}}{\sqrt{\bar\varepsilon}}\right) \right]
\right\}
\end{equation}
inner iterations where $R(\cdot;\cdot)$ is as defined in \eqref{eq:def_Rphi}.
\end{theorem}
\begin{proof}
It follows from the second statement following the AIPP method and the definition of $(\lambda,z^-,z, w,\varepsilon)$ in step~3 that the quintuple  $(\lambda,z^-,z, w,\varepsilon)$  satisfies the inclusion in \eqref{eq:ref4'}. Now, the bound in \eqref{eq:lastboundAIPP} follows by multiplying the bounds in Lemma~\ref{eq:AIPPcomplexinner}(a) and (b), and  adding the result to the bound in Lemma~\ref{eq:AIPPcomplexinner}(d).\end{proof}

Before stating the next result, we make two remarks about the above result. 
First, even though our main interest is in the case where $m \le M$ (see assumption (A.2)), bound \eqref{eq:lastboundAIPP} also holds for the case in which $m>M$.
Second, the AIPP version in which $\lambda = 1/(2m)$ yields the
the best complexity bound under the reasonable assumption that, inside the squared bracket in \eqref{eq:lastboundAIPP}, the first term is larger than the second one.

The following result describes the  inner iteration complexity of the AIPP method with $\lambda = 1/(2m)$
to compute approximate solutions of \eqref{eq:PenProb2Intro}  in the sense of \eqref{eq:ref4'''}.

\begin{corollary}\label{cor:AIPPref2} Assume that (A1)-(A3) hold and let a tolerance $\hat\rho>0$ be given. Also, let $(\lambda,z^-,z,w,\varepsilon)$
be the output obtained by the  AIPP method with inputs  $\lambda=1/(2m)$ and $(\bar\rho,\bar\varepsilon)$ defined as
\begin{equation}\label{eq:cor_complex1}
\quad (\bar\rho, \bar \varepsilon) := \left( \frac{\hat\rho}{4} \, ,  \frac{\hat\rho^2}{32(M+2m)} \right). 
\end{equation}
Then the following statements hold:
\begin{itemize}
\item[(a)] the AIPP method  terminates in at most
\begin{equation}\label{auxcomplex00}
\mathcal{O}\left\{\sqrt{\frac{M}{m} }\left(\frac{m^2 R(\phi;\lambda)} {\hat \rho^2}
+  \log^{+}_1\left(\frac{M}{m}\right)\right)\right\}
\end{equation}
inner iterations where $R(\cdot;\cdot)$ is  as in \eqref{eq:def_Rphi}.
\item[(b)]
if  $\nabla g$ is $M$-Lipschitz continuous, then the pair  $(\hat z, \hat v)=(z_g,v_g)$ computed  according to  \eqref{eq:def_zg} and \eqref{eq:def_vg}  
is a $\hat{\rho}$-approximate solution of \eqref{eq:PenProb2Intro}, i.e., \eqref{eq:ref4'''} holds. 
\end{itemize}
\end{corollary}
\begin{proof}
(a) This statement follows immediately from  Theorem~\ref{th:AIPPcomplexity} with $\lambda=1/(2m)$ and $(\bar\rho,\bar\varepsilon)$ as in  \eqref{eq:cor_complex1} and the fact that $m\leq M$ due to (A2).

 (b) First note that Theorem~\ref{th:AIPPcomplexity} implies that the AIPP output $(\lambda,z^-,z,w,\varepsilon)$ satisfies criterion \eqref{eq:ref4'} with  $(\bar\rho,\bar\varepsilon)$  as  in \eqref{eq:cor_complex1}.  Since \eqref{eq:cor_complex1} also implies that 
$$
\hat\rho= 2 \left[ \bar\rho+ \sqrt{2\bar\varepsilon \left(M+2m\right)} \right],
$$
the conclusion of (b)   follows from Proposition~\ref{prop:refapproxsol}(c) and the fact that $\lambda=1/2m$.
\end{proof}

\vgap
We now make a few remarks about the iteration-complexity bound \eqref{auxcomplex00}
and its relationship to  two other ones obtained in the literature under the reasonable assumption that
the  term $\mathcal{O}(1/\hat\rho^2)$ in \eqref{auxcomplex00} dominates the other one, i.e.,
bound \eqref{auxcomplex00} reduces to
\[
{\cal O} \left(\frac{m\sqrt{Mm} \, R(\phi;\lambda)}{\hat\rho^2}
\right).  
\]
First, using the definition of $R(\phi;\lambda)$, it is easy to see that the above bound is majorized by the one in \eqref{eq:boundIntro}
(see the proof of Corollary \ref{proprasc1}(a)).
Second, since the iteration-complexity bound for the composite gradient method with $\lambda =1/M$
is  ${\cal O}(M(\phi(z_0)-\phi_*)/\hat \rho^2)$ (see the discussion following Proposition \ref{prop:gradient method}), we conclude that \eqref{eq:boundIntro}, and hence \eqref{auxcomplex00},
is better than the first bound by a factor of $(M/m)^{1/2}$.
Third, 
bound \eqref{eq:boundIntro}, and hence \eqref{auxcomplex00},  is also better than the  one established in Corollary 2 of \cite{nonconv_lan16}
for an ACG method applied directly to
the nonconvex problem \eqref{eq:PenProb2Intro}, namely \eqref{eq:compl-lan},
by at least a factor of $(M/m)^{1/2}$. Note that the ACG method of \cite{nonconv_lan16} assumes that the diameter
$D_h$ of $\dom h$ is bounded while the AIPP method does not.   



\section{The QP-AIPP method}\label{sec:penaltymet}
This section presents the QP-AIPP method for obtaining approximate solutions of the
linearly constrained nonconvex composite optimization problem \eqref{eq:probintro} in the sense of \eqref{eq:approxOptimCond_Intro}.

Throughout this section, it is assumed that  \eqref{eq:probintro} satisfies the following conditions:

\begin{itemize}
\item [(C1)] $h\in \bConv{n}$, $A \ne 0$ and
${\cal F} :=\left\{ z \in \dom h : Az=b \right\}  \ne \emptyset$;
\item [(C2)] $f$ is a differentiable function on $\dom h$ and there exist scalars
$0 < m_f \le L_f$ such that for every $u,z \in \dom h$,
\begin{align}
\|\nabla f(z)-\nabla f(u)\| \leq L_f\|z-u\|, \label{eq:lips1} \\ 
\frac{-m_f}{2}\|u-z\|^{2}+\ell_f(u; z)\leq f(u); \label{eq:lips2} 
\end{align}
\item [(C3)] there exists $\hat c\geq 0$ such that  $\hat \varphi_{\hat c} > -\infty$ where
\begin{equation} \label{eq:penprob}
\hat \varphi_{c}  :=\inf_z \left \{ \varphi_c(z) := (f+h) (z)+\frac{c}{2}\|Az-b\|^2 : z \in \Re^n \right \}, \quad \forall c \in \Re;
\end{equation}
\end{itemize}

We make two remarks about conditions (C1)-(C3).
First, (C1) and (C3) imply that the optimal value of \eqref{eq:probintro} is finite but not necessarily achieved.
%
%
Second, (C3) is quite natural in the sense that the penalty approach underlying the QP-AIPP method would
not make sense without it.
Finally, \eqref{eq:lips1} implies that
\begin{equation}\label{ineq:concavity_penf1}
\frac{-L_f}{2}\|u-z\|^{2}\leq f(u)-\ell_f(u; z) \leq\frac{L_f}{2}\|u-z\|^{2},
\quad\forall z,u\in \dom h.
\end{equation}
and hence that \eqref{eq:lips2} automatically holds with $m_f=L_f$, i.e.,
\eqref{eq:lips2} is redundant when $m_f=L_f$. Our analysis in this section also considers 
the case in which a scalar $0<m_f<M_f$ satisfying \eqref{eq:lips2} is known.

Given a tolerance pair $(\hat \rho,\hat \eta)\in \Re^2_{++}$, a triple 
$(\hat z, \hat v, \hat p)$ is said to be a $(\hat \rho,\hat \eta)$-approximate solution
of \eqref{eq:probintro} if it satisfies \eqref{eq:approxOptimCond_Intro}.
Clearly, a $(\hat \rho,\hat \eta)$-approximate solution $(\hat z, \hat v, \hat p)$ for the case in which $(\hat \rho,\hat \eta) = (0,0)$ means
that 
$0 = \hat v \in \nabla f(\hat z) + \partial h(\hat z) + A^* \hat p$ and  $A \hat z=b$,
 and  hence that $(\hat z,\hat p)$ is a first-order stationary pair of \eqref{eq:probintro}.

The QP-AIPP method is essentially 
a quadratic penalty approach where the AIPP method is applied to the penalty subproblem \eqref{eq:penprob} associated with \eqref{eq:probintro}
for a fixed $c>0$ or for $c$ taking values on an increasing sequence $\{c_k\}$ converging to infinity.
Note that \eqref{eq:penprob}  is a particular case of \eqref{eq:PenProb2Intro} in which 
\begin{equation} \label{eq:gc}
g = g_c := f+\frac{c}2 \|A(\cdot)-b\|^2.
\end{equation}
 Moreover,
we easily see that \eqref{eq:lips2} and \eqref{ineq:concavity_penf1} imply that
$\nabla g_c$ satisfies condition \eqref{ineq:concavity_g}  with $(m, M)=(m_f, L_f+c\|A\|^2)$.

Lemmas \ref{lem:pensec_incl} and \ref{lem:penalty} below describe how a $\bar \rho$-approximate solution of \eqref{eq:penprob} in the sense of \eqref{eq:ref4'''}
yields  a $(\hat \rho,\hat \eta)$-approximate solution  of \eqref{eq:probintro} whenever $c$ is sufficiently large.
Lemma \ref{lem:Rc} introduces an important quantity associated with the penalized problem \eqref{eq:penprob} which plays
a fundamental role  in expressing the inner iteration complexity of the QP-AIPP method stated below for the case
in which $\dom h$ is not necessarily bounded (see Theorem \ref{th:adapAIPPPmet}). It also
establishes a few technical inequalities involving this quantity, one of which plays an important role in the
proof of Theorem \ref{th:adapAIPPPmet} and the statement of Lemma \ref{lem:penalty}.
%

\begin{lemma}\label{lem:pensec_incl}
Let $(c,\hat \rho)\in \Re^2_{++}$ be given and let $(\hat z,\hat v)$ be a $\hat\rho$-approximate solution of \eqref{eq:penprob} in the sense of \eqref{eq:ref4'''} with $g=g_c$ where $g_c$ is as in \eqref{eq:gc}. Then the triple $(\hat z,\hat v,\hat p)$ where $\hat p:=c(A \hat z-b)$ satisfies the inclusion and the first inequality in \eqref{eq:approxOptimCond_Intro}.
\end{lemma}
\begin{proof}
Since $(\hat z,\hat v)$ is a $\hat\rho$-approximate solution of \eqref{eq:penprob}, we have 
$\hat v \in \nabla g_c (\hat z)+\partial h(\hat z)$ and $\|\hat v\|\leq \hat \rho$.
Hence the result follows from  the definition of $\hat p$ and the fact that   $\nabla g_c(\hat z)=\nabla g(\hat z)+ A^*(c(A\hat z-b))= \nabla g(\hat z)+A^*\hat p$.
\end{proof}

The above result is quite general in the sense that it holds for any $c>0$. We will now show that,  by choosing $c$ sufficiently large, we can actually guarantee
that $(\hat z,\hat v,\hat p)$ of Lemma \ref{lem:pensec_incl} also satisfies the second inequality in \eqref{eq:approxOptimCond_Intro} as long as
$(\hat z,\hat v)$ is generated by an instance of the GIPP framework. We first establish the following technical result.

%
%

\begin{lemma}\label{lem:Rc}
Let $\hat c$ be as in (C3) and define 
\begin{equation}\label{def:Rc}
R_c(\lambda) :=   \inf \{ R (u;\varphi_{c}, \lam) : u \in {\cal F} \} \quad \forall (c,\lambda)\in \Re_+ ^2
\end{equation}
where  ${\cal F}$, $R (\cdot \, ;\cdot,\cdot)$ and $\varphi_c$ are as defined in (C1), \eqref{eq:def_Rphi}  and \eqref{eq:penprob}, respectively. 
Then, for every $c \ge \hat c  $ and $\lambda \ge \hat \lam \in \R_{+}$,  we have
\begin{align}
0 &\le R(u;\varphi_c,\lambda) \le R(u;\varphi_{\hat c}, \hat \lambda) < \infty \quad \forall u \in {\cal F}, \label{eq:Rc1} \\
0 &\leq  R_c(\lambda) \leq R_{\hat c}(\hat \lambda) < \infty. \label{eq:ineqRc}
\end{align} 
Moreover,  if  \eqref{eq:probintro} has an optimal solution $z^*$, then
\begin{equation} \label{eq:tempp}
R_c(\lam) \le \frac12 \| z_0-z^*\|^2 + (1-\sigma) \lam[ \hat \varphi_*- \hat \varphi_c]
\end{equation}
where $\hat \varphi_*$ denotes the optimum value of \eqref{eq:probintro}.
\end{lemma}
\begin{proof}
Using  \eqref{eq:penprob} and assumption (C3), it easy to see that for every $c \ge \hat c$, we have $\hat \varphi_c \ge \hat \varphi_{\hat c}>-\infty$ and
$\varphi_c(u)=\varphi_{\hat c}(u)=(f+h)(u)$ for every $u \in {\cal F}$.
Hence, the conclusion of the lemma follows immediately from \eqref{def:Rc} and the definition of $R(\cdot \, ;\cdot,\cdot)$ in  \eqref{eq:def_Rphi}.
\end{proof}


We are now ready to describe  the feasibility behavior of a GIPP instance applied to \eqref{eq:penprob}.
\begin{lemma}\label{lem:penalty}
Assume that  $\{(\lambda_k, z_k,\tilde v_k,\tilde \varepsilon_k)\}$ is a  sequence
 generated by an instance of the GIPP framework with input $\sigma \in (0,1)$ and $z_0\in \dom h$  and with $\phi=\varphi_c$ for some
$c>\hat c$ where $\hat c$ is as in (C3) and $\varphi_c$ is as in \eqref{eq:penprob}. 
Also, let $\hat \eta\in \Re_{++}$ be given and define
\begin{equation}\label{def:Rhatc}
T_{\hat \eta} (\lambda) :=  \frac{2R_{\hat c}(\lambda)}{\hat \eta^2  (1-\sigma)\lambda}  + \hat c   \qquad \forall \lambda\in \Re_{++}                                     
\end{equation}
where $R_{\hat c}(\cdot)$ is as defined in \eqref{def:Rc}.
 Then for every 
$ \hat z \in \Re^n$ such that $\varphi_c(\hat z)\leq\varphi_c(z_1)$, we have
\begin{equation}\label{eq:pre_feasib}
\|A \hat z -b\|^2 \le \frac{[T_{\hat \eta}(\lambda_1) - \hat c]\hat\eta^2}{c-\hat c}.
\end{equation}
As a consequence, if $c \ge T_{\hat \eta}(\lambda_1) $ then $$\|Az_k-b\|\leq \hat \eta, \quad \forall  k\geq 1.$$
\end{lemma}
\begin{proof} 
First note that the definitions of $\varphi_c$ and $\hat \varphi_c$ in \eqref{eq:penprob} imply that for every $c>0$,
\begin{align}\label{eq:properites_varphi}
\varphi_c(u) &= \varphi_{\hat c} (u) + \frac{c-\hat c }2 \|A u-b\|^2 \ge \hat\varphi_{\hat c}  + \frac{c-\hat c }2 \|A u-b\|^2 \quad
\forall u \in \Re^n.
\end{align}
 Now, let $ \hat z\in \Re^n$ be such that $\varphi_c(\hat z)\leq\varphi_c(z_1)$.
Lemma~\ref{lem:phikd0} with  $\phi=\varphi_c$ and
 $k=1$, the previous inequality on $\hat z$, and  \eqref{eq:properites_varphi} with $u=\hat z$, then imply that for every $u\in \mathcal{F}$,
\begin{align*}
&\frac{\|z_0-u\|^2}{2(1-\sigma)\lambda_1}+\varphi_{c}(u)\ge  \varphi_c(z_1)\ge  \varphi_c(\hat z)
\ge \hat\varphi_{\hat c} +\frac{c-\hat c}{2}\|A \hat z-b\|^2.
\end{align*}
Since $\varphi_c(u)=\varphi_{\hat c}(u)$ for every $u \in {\cal F}$, it then follows from the above inequality and  the definition of $R(\cdot \, ;\cdot,\cdot)$ in \eqref{eq:def_Rphi} that  
\[
\|A \hat z-b\|^2\leq \frac{2R(u;\varphi_{\hat c},\lambda_1)}{(c-\hat c)(1-\sigma)\lambda_1} \quad \forall u \in {\cal F}.
\] 
Since the above inequality holds for every $u\in \mathcal{F}$, it then follows from \eqref{def:Rc} and the  definition of $T_{\hat \eta}(\cdot)$ in \eqref{def:Rhatc} that
the first conclusion of the lemma holds.
Now, since by assumption $\{(\lambda_k, z_k,\tilde v_k,\tilde \varepsilon_k)\}$  is generated by an instance of
the GIPP framework with $\phi=\varphi_c$, it follows from \eqref{eqLemaux:rasc}  that
 $\varphi_c(z_k)\leq\varphi_c(z_1)$ for every $k\geq 1$, and hence that the second conclusion of the lemma follows immediately from the first one together with
the assumption that $c \ge T_{\hat \eta}(\lambda_1)$.
\end{proof}

We now make some remarks about the above result. First, it does not assume that ${\cal F}$, and hence $\dom h$, is bounded.
Also, it does not even assume that \eqref{eq:probintro} has an optimal solution.
Second, it  implies that all iterates (excluding the starting one) generated by an instance of the GIPP framework applied to
\eqref{eq:penprob} satisfy the feasibility requirement (i.e., the last inequality) in \eqref{eq:approxOptimCond_Intro} as long as $c$ is sufficiently large, i.e.,
$c \ge T_{\hat \eta}(\lambda_1)$ where $T_{\hat \eta}(\cdot)$ is as in \eqref{def:Rhatc}.
Third, since the quantity $R_{\hat c}(\lam_1)$, which appears in the definition of $T_{\hat \eta}(\lambda_1)$ in \eqref{def:Rhatc},
is difficult  to estimate, a simple way of choosing a  penalty parameter $c$ such that $c \ge T_{\hat \eta}(\lambda_1)$ is not apparent. The QP-AIPP method described below
solves instead a sequence of penalized subproblems \eqref{eq:penprob} for increasing values of $c$ (i.e., updated according to $c \leftarrow 2c$). Moreover, despite solving a
sequence of penalized subproblems,
it is shown that its overall ACG iteration complexity is the same as the one for the ideal method corresponding to solving \eqref{eq:penprob} with
$c=T_{\hat \eta}(\lambda_1)$.

We are ready to state the QP-AIPP method.

%
%
%

\noindent\rule[0.5ex]{1\columnwidth}{1pt}

QP-AIPP Method 

\noindent\rule[0.5ex]{1\columnwidth}{1pt}

\begin{itemize}
\item [(0)]  Let  $z_0 \in \dom h $, $\sigma \in (0,1)$, $L_f$ satisfying \eqref{ineq:concavity_penf1}, $m_f$ satisfying \eqref{eq:lips2}, and  a tolerance pair  $(\hat \rho,\hat \eta)\in \Re^2_{++}$ be given, and set
\begin{equation}\label{eq:lam_c_def}
 \lambda=\frac{1}{2m_f}, \quad c=\hat c + \frac{L_f}{\|A\|^2};
 \end{equation} 
\item [(1)]  apply the AIPP method with inputs $z_0$, $\sigma$, $\lambda$,
 \begin{equation}\label{eq:mMpenmet}
(m, M) = (m_f,L_f+c\|A\|^2), 
\end{equation} 
and $ (\bar\rho, \bar \varepsilon)$ as in \eqref{eq:cor_complex1}
to find a $(\bar \rho, \bar \varepsilon)$-prox approximate solution $(\lam, z^-, z,w,\varepsilon)$
of  problem \eqref{eq:PenProb2Intro}  (according to \eqref{eq:ref4'}) with $g:=g_c$ and $g_c$  as in \eqref{eq:gc};
\item [(2)] use $(\lam,z^-,z,w,\varepsilon)$ to compute  $(z_g,v_g)$ as  in  Proposition~\ref{prop:refapproxsol} with $g=g_c$; 
\item [(3)] if $\|A z_g-b\|>  \hat\eta$ then set  $c=2c$ and go to (1); otherwise, stop and output $(\hat z,\hat v,\hat p)=(z_g,v_g,c(A z_g-b))$. 
\end{itemize} 
\noindent\rule[0.5ex]{1\columnwidth}{1pt}

\gap 

Every loop of the  QP-AIPP method  invokes in its step 1 the AIPP  method of Subsection~\ref{sec:AIPPmet} to compute a
$(\bar \rho, \bar \varepsilon)$-prox approximate solution
of  \eqref{eq:PenProb2Intro}. 
The latter method in turn uses the ACG method of Subsection~\ref{subsec:Nesterov's-Method} as a subroutine in its implementation
(see step~1 of the AIPP method).
For simplicity, we refer to all ACG iterations performed during calls to the ACG method as inner iterations.

We now make a few remarks about the QP-AIPP method.
First, it follows from Corollary~\ref{cor:AIPPref2}(b) that the pair $(z_g,v_g)$ is a $\hat\rho$-approximate solution of \eqref{eq:penprob} in the sense of \eqref{eq:ref4'''}
with $g=g_c$ and $g_c$  as in \eqref{eq:gc}.
As a consequence, Lemma~\ref{lem:pensec_incl} implies that
the output $(\hat z, \hat v,\hat p)$ satisfies the inclusion and the first inequality in \eqref{eq:approxOptimCond_Intro}.
Second, since  $(\lam,z^-,z,w,\varepsilon)$ computed at step 1 is an iterate of the AIPP method, which in turn is a special instance of
the GIPP framework, and $\phi_c(z_g) \le \phi_c(z)$ due to \eqref{ineq:qfdeltaf}
of Lemma~\ref{lem:approxsolreps}, we conclude from Lemma \ref{lem:penalty} that $\hat z=z_g$ satisfies \eqref{eq:pre_feasib}. Third, since
every loop of the QP-AIPP method doubles $c$, the condition $c > T_{\hat \eta}(\lambda_1)$ will be eventualy satisfied. Hence, in view of
the previous remark, the $z_g$ corresponding to this $c$ will satisfy the feasibility condition $\|A z_g - b\| \le \hat \rho$ and the
QP-AIPP method will stop in view of its stopping criterion in step 3.
Finally, in view of the first and third remarks, we conclude that the QP-AIPP method terminates in step 3 with a triple $(\hat z,\hat v,\hat p)$
satisfying \eqref{eq:approxOptimCond_Intro}.

The next result derives a bound on the overall number of inner iterations of the  quadratic penalty AIPP method to
obtain an approximate solution  of  \eqref{eq:probintro} in the sense of \eqref{eq:approxOptimCond_Intro}. 
\begin{theorem}\label{th:adapAIPPPmet}
Let $\hat c$ be as in (C1) and define
\begin{equation}\label{eq:Theta}
\lambda:=\frac{1}{2m_f},\quad T_{\hat \eta} :=  \frac{2R_{\hat c}(\lambda)}{\hat \eta^2  (1-\sigma)\lambda}  + \hat c, \quad \Theta :=\frac{L_f + T_{\hat \eta} \|A\|^2}{m_f}
\end{equation}
where $R_{\hat c}(\lambda)$ is as in  \eqref{def:Rc}.
Then, the  QP-AIPP method outputs a triple $(\hat z,  \hat v,\hat p)$ satisfying
\eqref{eq:approxOptimCond_Intro}
in a total number of inner iterations bounded by
\begin{equation}\label{bound:adaptmet}
\mathcal{O}\left( \sqrt{\Theta} \left[\frac{m_f^2 R_{\hat c}(\lam)}{\hat \rho^2}+\log^+_1(\Theta)\right]\right).
\end{equation}


\end{theorem}
\begin{proof}
Let $c_1 := \hat c + L_f/\|A\|^2$. Noting the stopping criterion in step~3, using the second remark preceding the theorem and the fact that \eqref{eq:lam_c_def} implies  $c=c_l:=2^{l-1}c_1$ at the $l$-th  loop of the QP-AIPP method, we conclude that the QP-AIPP method stops in at most $\bar{l}$ loops where $\bar{l}$ is
the first index $l \ge 1$ such that  
$
2^{l-1}c_1 >  T_{\hat \eta}.
$
We claim that
\begin{equation} \label{eq:claim}
\sum_{l=1}^{\bar l}  \sqrt{\frac{L_f + c_l \|A\|^2}{m_f}}   \le {\cal O} \left(\sqrt{ \frac{L_f + T_{\hat \eta}\|A\|^2}{m_f} }\right) = {\cal O}(\sqrt{\Theta}).
\end{equation}
Before establishing the above claim, we will use it to show that the conclusion of the theorem holds.
Indeed,
%
%
first note that  the definition of $c_1$ and the above definition of $c_l$ imply that $c_l \ge c_1 \ge \hat c$ for every $l \ge 1$.
Hence, it follows from the second inequality in \eqref{eq:ineqRc} with $(c,\hat \lambda)=(c_l,\lambda)$  that $R_{c_l}(\lam) \le R_{\hat c}(\lam)$.
Since \eqref{eq:def_Rphi} and \eqref{def:Rc} easily imply that $R(\varphi_{c_l},\lam) \le R_{c_l}(\lambda) $, 
we then conclude that $R(\varphi_{c_l}, \lam) \le R_{\hat c}(\lam)$.
The latter conclusion, \eqref{eq:claim} 
and Corollary~\ref{cor:AIPPref2}(a) with  $\phi=\varphi_{c_l}$ and $(m,M)$ as in \eqref{eq:mMpenmet} then imply
that the number of inner iterations during the $l$-th  loop of the  QP-AIPP method  is bounded by 
\begin{equation}\label{innercomplexadapAIPPP}
\mathcal{O}\left(\sqrt{\frac{L_f + c_l \|A\|^2}{m_f}} \left[\frac{m_f^2 R_{\hat c}(\lambda)}{\hat \rho^2}+
\log^+_1 \left( \Theta \right) \right] \right).
\end{equation}
Hence, the total number of inner iterations performed by the QP-AIPP method is bounded by the sum of  the previous bound over $l=1,\ldots,\bar l$ which is
equal to \eqref{bound:adaptmet} in view of \eqref{eq:claim}.

We will now show that \eqref{eq:claim} holds.   If $\bar{l}=1$ then 
it follows from the definitions of $\Theta$ in  \eqref{eq:Theta}  and $c_1$ in the beginning of the proof that
\begin{equation}\label{ineq:theta_subprb}
\frac{c_1\|A\|^2}{m_f}= \frac{L_f + \hat c\|A\|^2}{m_f} \leq \frac{L_f + T_{\hat \eta}\|A\|^2}{m_f} = \Theta
\end{equation}
and hence \eqref{eq:claim} holds.
Consider now the case in which $\bar{l}>1$. Using the fact that   $c_l=2^{l-1}c_1$ together with the first equality in \eqref{ineq:theta_subprb}, we have that $c_l\|A\|^2/m_f \geq c_1|A\|^2/m_f \geq L_f / m_f $ and hence
\begin{align}\label{ineq:theta_large_lbar}
&\sum_{l=1}^{\bar l}  \sqrt{\frac{L_f + c_l \|A\|^2}{m_f}} 
\le   \sqrt{\frac{2 c_1 \|A\|^2}{m_f}}\sum_{l=1}^{\bar l} (\sqrt{2})^{l-1} 
\leq\mathcal{O}\left(\sqrt{\frac{c_1\|A\|^2} {m_f}}\sqrt{2}^{\bar{l}}\right).
\end{align}
Using the definition of $\bar l$ and the fact that $\bar l > 1$, we easily see that $2^{\bar l-1}c_1 \le 2 T_{\hat \eta}$ and hence that $(\sqrt{2})^{\bar l} \leq 2(T_{\hat \eta}/c_1)^{1/2}$. The last inequality together with \eqref{ineq:theta_large_lbar} and the definition of $\Theta$ in  \eqref{eq:Theta} then imply \eqref{eq:claim}.
\end{proof}

Before ending this section, we make three remarks about Theorem~\ref{th:adapAIPPPmet}.
First, \eqref{eq:tempp} implies the quantity $R_{\hat c}(\lam)$ admits the upper bound
\begin{equation*}
R_{\hat c}(\lam) \le \frac12 \hat d_0^2 + (1-\sigma) \lam[ \hat \varphi_*- \hat \varphi_{\hat c}]
\end{equation*}
where $\hat d_0 := \left\{\|z_0 - z_*\| : z_* \text{ is an optimal solution of \eqref{eq:probintro}} \right\}$. Second, in terms of the tolerance pair $(\hat \rho,\hat \eta)$ only, the iteration-complexity bound  \eqref{bound:adaptmet} reduces
to $\mathcal{O}\left(1/(\hat \rho^2\hat \eta)\right)$ for an arbitrary initial point $z_0 \in \dom h$.
Third, the iteration-complexity bound \eqref{bound:adaptmet}  is almost the same as the one corresponding to the case in which $T_{\hat\eta}=T_{\hat\eta}(\lam)$ as in \eqref{def:Rhatc} is known and the penalty parameter is set to $c=T_{\hat \eta}$, namely,
$$ 
{\cal O} \left (\sqrt{\Theta} \left  [   \frac{m_f^2 R(\varphi_{c},\lam) }{\hat \rho^2}  + \log^+_1(\Theta) \right ] \right)
= {\cal O} \left (\sqrt{\Theta} \left  [   \frac{m_f^2 R_{c} (\lam) }{\hat \rho^2}  + \log^+_1(\Theta) \right ] \right)
$$
which follows as a consequence of
Corollary~\ref{cor:AIPPref2}  with $M$ as in \eqref{eq:mMpenmet}.
Note that the two bounds differ only in that the quantity $R_{c}(\lam)$,
which appears in the above bound, may be strictly smaller than the quantity $R_{\hat c}(\lam)$ in \eqref{bound:adaptmet} (see \eqref{eq:ineqRc}). 




\section{Computational Results}
\label{sec:computResults}

   The goal of this section is to present a few computational results that show the performance of the AIPP method, which would consequently assess the performance of the QP-AIPP method. 
In particular, while the experiments are limited in the sense that they do not directly test the actual behavior of the QP-AIPP method, they can be considered as examples of subproblems in the 
execution of the penalty-based method. The AIPP method is benchmarked against two other nonconvex optimization methods, namely the projected gradient (PG) method and the accelerated 
gradient (AG) method recently proposed and analyzed in \cite{nonconv_lan16}.

Several instances of the quadratic programming (QP)  problem
\raggedbottom
\begin{equation}\label{testQPprob}
\min\left\{ g(z):=-\frac{\xi}{2}\|DBz\|^{2}+\frac{\tau}{2}\|Az-b\|^{2}:z\in\Delta_{n}\right\}
\end{equation}
were considered where $A\in \Re^{l\times n}$,  $B\in \Re^{n\times n}$,  $D\in \Re^{n\times n}$ is a diagonal matrix, $b\in \Re^{ l\times1}$, $(\xi,\tau)\in \Re^2_{++}$,  and
$\Delta_{n}:=\left\{ z\in\Re^n:\sum_{i=1}^{n}z_{i}=1, \;\; z_i\geq0\right\}$.
More specifically, we set the dimensions to be $(l,n)=(20,300)$.  We also generated the entries of $A,B$ and $b$ by sampling from the uniform distribution ${\cal U}[0,1]$ 
and the diagonal entries of $D$ by sampling from the discrete uniform distribution ${\cal U}\{1,1000\}$.  By appropriately choosing the scalars  $\xi$ and $\tau$,
the instance corresponding to a pair of parameters $(M,m)\in\Re^2_{++}$
was generated so that $M=\lambda_{\max}(\nabla^{2}g)$ and $-m=\lambda_{\min}(\nabla^{2}g)$ where $\lambda_{\max}(\nabla^{2}g)$ and $\lambda_{\min}(\nabla^{2}g)$ denote
 the largest and smallest eigenvalues of the Hessian of $g$ respectively.
The parameters $(\lambda,\sigma)$ were set to be  $(0.9/m, 0.3)$. The AIPP, PG, and AG methods were implemented in 
MATLAB 2016a scripts and were run on Linux 64-bit machines each containing Xeon E5520 processors and at least 8 GB of memory. 

All three methods use the centroid of the set
$\Delta_{n}$ as the initial starting point $z_0$  and were run until a pair $(z,v)$ was generated satisfying the condition
\begin{equation}
v\in \nabla g(z)+N_{\Delta_n}(z), \qquad  \frac{\|v\|}{\|\nabla g(z_{0})\|+1}\leq \bar \rho\label{eq:comp_term}
\end{equation}
for a given tolerance $\bar \rho>0$. Here, $N_{X}(z)$ denotes the normal cone of $X$ at $z$, i.e. $N_X(z)=\{u\in \Re^n: \langle u, \tilde z-z\rangle \leq 0,\; \forall \tilde z \in X\}.$ 
The results of the two tables below were obtained with $\bar \rho=10^{-7}$. 
They present results for different choices of the curvature pair $(M,m)$.  Each entry  in the $\bar{g}$-column is the value of the objective function of \eqref{testQPprob}
at the last iterate generated by each method. Since they are approximately the same for all three methods, only one value is reported.
 The bold numbers in each table highlight the algorithm that performed the most efficiently in terms of total number of iterations. 
It should be noted that both AIPP and PG methods perform a single projection step per iteration while the AG method performs two.


\begin{table}[H]
\begin{centering}
\begin{tabular}{|>{\centering}p{1.5cm}>{\centering}p{1.5cm}|>{\centering}m{2cm}|>{\centering}p{1.6cm}>{\centering}p{1.6cm}>{\centering}p{1.6cm}|}
\hline 
\multicolumn{2}{|c|}{Size} & \multirow{2}{2cm}{\centering{}$\bar{g}$} & \multicolumn{3}{c|}{Iteration Count}\tabularnewline
{\small{}$M$} & {\small{}$m$} &  & {\small{}PG} & {\small{}AG} & {\small{}AIPP}\tabularnewline
\hline 
{\small{}16777216} & {\small{}16777216} & {\small{}-2.24E+05} & {\small{}5445} & \textbf{\small{}374} & {\small{}14822}\tabularnewline
{\small{}16777216} & {\small{}1048576} & {\small{}-3.83E+04} & {\small{}7988} & \textbf{\small{}4429} & {\small{}6711}\tabularnewline
{\small{}16777216} & {\small{}65536} & {\small{}-4.46E+02} & {\small{}91295} & \textbf{\small{}22087} & {\small{}24129}\tabularnewline
{\small{}16777216} & {\small{}4096} & {\small{}4.07E+03} & {\small{}80963} & {\small{}26053} & \textbf{\small{}5706}\tabularnewline
{\small{}16777216} & {\small{}256} & {\small{}4.38E+03} & {\small{}82029} & {\small{}20371} & \textbf{\small{}1625}\tabularnewline
{\small{}16777216} & {\small{}16} & {\small{}4.40E+03} & {\small{}81883} & {\small{}20761} & \textbf{\small{}2308}\tabularnewline
\hline 
\end{tabular}
\par\end{centering}
\caption{Numerical results with $\sigma=0.3$ and $\lambda=0.9/m$\label{tab:t1}}
\end{table}
\vspace{-5mm}

\begin{table}[H]
\begin{centering}
\begin{tabular}{|>{\centering}p{1.5cm}>{\centering}p{1.5cm}|>{\centering}m{2cm}|>{\centering}p{1.6cm}>{\centering}p{1.6cm}>{\centering}p{1.6cm}|}
\hline 
\multicolumn{2}{|c|}{Size} & \multirow{2}{2cm}{\centering{}$\bar{g}$} & \multicolumn{3}{c|}{Iteration Count}\tabularnewline
{\small{}$M$} & {\small{}$m$} &  & {\small{}PG} & {\small{}AG} & {\small{}AIPP}\tabularnewline
\hline 
{\small{}4000} & {\small{}1} & {\small{}9.68E-01} & {\small{}80560} & {\small{}24813} & \textbf{\small{}5752}\tabularnewline
{\small{}16000} & {\small{}1} & {\small{}4.11E+00} & {\small{}77813} & {\small{}24861} & \textbf{\small{}2830}\tabularnewline
{\small{}64000} & {\small{}1} & {\small{}1.67E+01} & {\small{}82000} & {\small{}20373} & \textbf{\small{}1621}\tabularnewline
{\small{}256000} & {\small{}1} & {\small{}6.71E+01} & {\small{}81929} & {\small{}20767} & \textbf{\small{}1942}\tabularnewline
{\small{}1024000} & {\small{}1} & {\small{}2.68E+02} & {\small{}81882} & {\small{}20761} & \textbf{\small{}2297}\tabularnewline
{\small{}4096000} & {\small{}1} & {\small{}1.07E+03} & {\small{}81871} & {\small{}20759} & \textbf{\small{}2083}\tabularnewline
\hline 
\end{tabular}
\par\end{centering}
\caption{Numerical results with $\sigma=0.3$ and $\lambda=0.9/m$\label{tab:t2}}
\end{table}
\vspace{-5mm}


From the tables, we can conclude that if the curvature ratio $M/m$ is sufficiently large then the AIPP method performs fewer iterations than the PG and the AG methods.
This indicates that the QP-AIPP, which is based on the AIPP method, might be a promising approach towards solving linearly constrained
nonconvex optimization problems. This is due to the fact that the curvature ratios of the penalty subproblems grow substantially as $c$ increases
and, as a result, 
AIPP can efficiently solve them.
On the other hand, AIPP does not do well on instances whose associated curvature ratio is small. However, preliminary computational experiments seem to indicate
that a variant of AIPP can also efficiently solve instances with small curvature ratios by significantly choosing $\lambda$ much larger than $0.9/m$.
Since this situation is not covered by the theory presented in this paper, we are not reporting these results in this paper,
opting instead to leave this preliminary investigation for a future work.

\section{Concluding remarks} \label{sec:remarks}


Paper \cite{ProxAugLag_Ming} proposed a linearized version of the augmented Lagrangian method to solve \eqref{eq:probintro} but assumes the strong condition (among a few others) that $h=0$,
which most important problems arising in applications do not satisfy.
To circumvent this technical issue, \cite{Jiang2018} proposed a penalty ADDM  approach which
introduces an artificial variable $y$ in (1) and then penalizes $y$ to obtain
the penalized problem
\begin{equation} \label{eq:penpr}
\min \left \{ f(z) + h(z) + \frac{c}2 \|y\|^2 : Ax+ y =b \right \},
\end{equation}
which is then solved by a two-block  ADMM. Since \eqref{eq:penpr} satisfies the assumption that
its $y$-block objective function component has Lipschitz continuous gradient everywhere and its $y$-block coefficient matrix
is the identity, an iteration-complexity of the two-block ADDM for solving \eqref{eq:penpr},
and hence \eqref{eq:probintro},  can be established. More specifically, it has been shown in Remark 4.3 of \cite{Jiang2018}  that the overall number of composite gradient steps performed by the aforementioned two-block ADMM penalty scheme to obtain a triple $(\hat z,\hat v,\hat p)$
satisfying \eqref{eq:approxOptimCond_Intro} is bounded by ${\cal O}(\hat \rho^{-6})$
under the assumptions that $\hat \eta = \hat \rho$, the level sets of $f+h$ are bounded and the initial triple
$(z_0,y_0,p_0)$ satisfies $(y_0, p_0)=(0, 0)$,  $Az_0=b$ and $z_0 \in \dom h$.

Note that  the last complexity bound is derived under a boundedness assumption and
is worse than the one obtained in this paper  for the QP-AIPP method, namely ${\cal O}(\hat \rho^{-2} \hat \eta^{-1})$,
without any boundedness assumption.
Moreover, in contrast to the complexity of the QP-AIPP which is established for an arbitrary infeasible point $z_0 \in \dom h$,
the complexity bound  of the aforementioned two-block ADMM penalty scheme assumes that $z_0$ is feasible for \eqref{eq:probintro}.
In fact, as far as we know, QP-AIPP is the first method for solving \eqref{eq:probintro} from an infeasible starting point
with a guaranteed complexity bound under the general assumptions considered in this paper.




\appendix
\section{Proof of Proposition~\ref{prop:refapproxsol}}\label{app:approxsollemma}


We first state two technical lemmas before giving the proof of Proposition~\ref{prop:refapproxsol}.

\begin{lemma}\label{lem:approxsolreps}
Assume that $h\in\bConv{n}$, 
$z \in \dom h$ and $f$ is a differentiable function on $\dom h$ which, for some $L > 0$, satisfies
\begin{equation}\label{uppercurvature_f}
f(u)- \ell_f(u; z) \leq\frac{L}{2}\|u-z\|^2,\qquad \forall u \in \dom h,
\end{equation}
and define
\begin{align}
z_f = z(z;f)  &:= \argmin_u \left\{ \ell_f(u;z)  + h(u) + \frac{L}{2} \|u-z\|^2  \right \}, \label{eq:def_zpz} \\
 q_f = q(z;f) &:= L(z-z_f),  \label{eq:def_qf} \\
\delta_f = \delta(z;f) &:= h(z) - h(z_f) - \inner{q_f - \nabla f(z) }{z-z_f}.\label{eq:def_deltaf} 
\end{align}
Then, there hold
\begin{align}
& q_f \in \nabla f(z) + \partial h(z_f), \quad q_f \in \nabla f(z) + \partial_{\delta_f} h(z), \quad \delta_f \ge 0,\label{inclusion_qf}
\\
& (q_f,\delta_f) = \argmin_{(r,\varepsilon)} \left \{ \frac{1}{2L} \| r\|^2 + \varepsilon : r \in \nabla f(z) + \partial_{\varepsilon} h(z) \right\},\label{argmin_qfdeltaf}
\\
 & \delta_f+ \frac{1}{2L} \|q_f\|^2 \le  (f+h)(z) - (f+h)(z_f). \label{ineq:qfdeltaf}
\end{align} 
\end{lemma}
\begin{proof}
We first show that \eqref{inclusion_qf} holds.
The optimality condition for \eqref{eq:def_zpz} and the definition of $q_f$ in  \eqref{eq:def_qf} immediately yield the first inclusion in \eqref{inclusion_qf}.
Hence, it follows from Proposition~\ref{prop:transpForm} and the definition of $\delta_f$ in \eqref{eq:def_deltaf} that the second inclusion and the inequality in \eqref{inclusion_qf} also hold. 

We now show that \eqref{argmin_qfdeltaf} holds.
Clearly, the second inclusion in \eqref{inclusion_qf} implies that $(q_f,\delta_f)$ is feasible to \eqref{argmin_qfdeltaf}. 
Assume now that $(r,\varepsilon)$ satisfies $r \in \nabla f(z) + \partial_{\varepsilon} h(z)$, or equivalently,
$$
h(u)\geq h(z)+\langle r- \nabla f(z), u-z\rangle-\varepsilon\quad \forall u\in \Re^n.
$$
Using the above inequality
with $u=z_f$ and the definitions of  $q_f$ and $\delta_f$ given in  \eqref{eq:def_qf} and \eqref{eq:def_deltaf}, respectively,  we then conclude that
\begin{align*}
\delta_f+\frac{\|q_f\|^2}{2L}  
&= h(z) - h(z_f) - \inner{\nabla f(z) }{z_f-z}- \frac{L}{2} \|z-z_f\|^2\\
&\leq -\langle r, z_f-z\rangle +\varepsilon  - \frac{\|q_f\|^2}{2L}= \frac1L \langle r, q_f \rangle +\varepsilon - \frac{\|q_f\|^2}{2L}\\&\leq \frac{1}{2L}\|r\|^ 2 +\frac{1}{2L}\|q_f\|^2+\varepsilon- \frac{\|q_f\|^2}{2L}=\frac{1}{2L}\|r\|^ 2 +\varepsilon
\end{align*}
where the last inequality is due to Cauchy Schwarz inequality and $2ab\leq a^2+b^2$, for every $a,b \in \Re$. Hence  \eqref{argmin_qfdeltaf} holds.
%
Finaly, to see that \eqref{ineq:qfdeltaf} holds, note that
the last relation, the definition of $\ell_f(\cdot;z)$ in \eqref{eq:defell} and inequality \eqref{uppercurvature_f} with $u=z_f$ imply  that
\begin{align*}
\delta_f+\frac{\|q_f\|^2}{2L}  &=(f+h)(z) - (f+h) (z_f) + [ f(z_f) - \ell_f(z_f;z) ]
- \frac{L}{2} \|z-z_f\|^2\\
&\leq (f+h) (z)-(f+h) (z_f).
\end{align*}
\end{proof}

\vgap
\begin{lemma}\label{lem:approxsolff}
Assume that $h\in \bConv{n}$, $z \in \dom h$  and $g$ is a differentiable function on $\dom h$ which, for some $M > 0$, satisfies
\eqref{uppercurvature_f} with $(f,L)$ replaced by $(g,M)$.
Let $\lambda>0$, $(z^-,z, w,\varepsilon) \in \Re^n \times \dom h\times \Re^n \times  \Re_+$ and $\rho>0$ be such that
\begin{equation} \label{inclusion_proxsolfg}
w \in \partial_\varepsilon \left( g + h + \frac{1}{2\lam} \|\cdot-z^-\|^2 \right)(z),
\quad \left \| \frac{1}\lam (z^--z) + w \right\| \le \rho
\end{equation}
 and   set
\begin{align} \label{eq:def_Lfzp}
&L =M + \lam^{-1}, \quad 
f(\cdot)=g(\cdot ) +\frac{1}{2\lam} \|\cdot -z^-\|^2-\langle w, \cdot\rangle, \\ 
& (z_f,q_f,\delta_f) = (z(z;f), q(z;f), \delta(z;f)), \nonumber
\end{align}
 where the quantities
$z(z;f)$, $q(z;f)$ and $\delta(z;f)$ are defined in \eqref{eq:def_zpz}, \eqref{eq:def_qf} and \eqref{eq:def_deltaf}, respectively.
Then, the following statements hold:

\begin{itemize}
\item[(a)] the pair $(v,\delta_f)$ where
\begin{equation}\label{def:vfg}
 v:= q_f + \frac{z^--z}{\lambda} + w ,
 \end{equation}
satisfies
\begin{align}
\label{eq:inclusion_vfg}
& v\in \nabla g(z) + \partial_{\delta_f} h (z),
\quad 0\leq\delta_f \leq \varepsilon; \\
\label{eq:inequality_vfg}
& \|v\|^2 + 2 \left( M+\lam^{-1} \right) \delta_f \le  \left  [ \rho + \sqrt{2(M+\lam^{-1})\varepsilon} \right ]^2;
\end{align}
\item[(b)]
if $\nabla g$ is $M$-Lipschitz continuous, then the pair $(z_f,v_f)$ where
\begin{equation}\label{def:vvff}
v_f:=v + \nabla g(z_f) - \nabla g(z) 
\end{equation}
satisfies
\begin{equation} \label{eq:incl_vff}
v_f \in \nabla g(z_f) + \partial h (z_f), \quad  \| v_f \| \le \rho + 2\sqrt{2 (M+\lambda^{-1}) {\varepsilon}}.
\end{equation}
\end{itemize}

\end{lemma}
\begin{proof}
(a) First note that the pair $(f,L)$ defined in \eqref{eq:def_Lfzp} satisfies \eqref{uppercurvature_f}
and that $\lambda$ and $(z^-,z,w,\varepsilon)$ satisfy the inclusion in \eqref{inclusion_proxsolfg} if and only if 
$
0 \in \partial_\varepsilon( f+h ) (z),
$
or equivalently,   $(f+h)(u) \ge (f+h)(z) - \varepsilon$ for every $u$.
In particular, the latter inequality with $u=z_f$ implies that
$
(f+h)(z) - (f+h)(z_f)\leq  \varepsilon. 
$
Hence, combining the last inequality, Lemma~\ref{lem:approxsolreps} with $(f,L)$ as in \eqref{eq:def_Lfzp}, and the definition of $v$ given in \eqref{def:vfg}, we conclude that
the relations in \eqref{eq:inclusion_vfg} hold and
\begin{equation}\label{requiredIneq1ff}
\| q_f\|^2 + 2 (M+ \lam^{-1}) \delta_f \le 2(M+\lambda^{-1})\varepsilon. 
\end{equation}
Now, the inequality in \eqref{inclusion_proxsolfg}, definition of $v$ in \eqref{def:vfg} and the triangle inequality for
norms  imply  that
$\|v\| \le \rho + \|q_f\|$, and hence, in view of \eqref{requiredIneq1ff}, that
\begin{align*}
 \|v\|^2 + 2  (M+ \lam^{-1}) \delta_f & \le 
\left[ \|q_f\|^2 + 2 (M+ \lam^{-1}) \delta_f \right] + 2 \rho \|q_f\| + \rho^2 \\
& \le 2(M+ \lam^{-1}) \varepsilon + 2 \rho \sqrt{2 (M+ \lam^{-1}) \varepsilon} + \rho^2 \\
&= \left[ \rho + \sqrt{2 (M+ \lam^{-1}) \varepsilon} \right]^2,
\end{align*}
showing that \eqref{eq:inequality_vfg}  holds.

\noindent(b) The inclusion in \eqref{eq:incl_vff} follows immediately from the first inclusion in \eqref{inclusion_qf} and definition of $v_f$ in \eqref{def:vvff}.
Finally, using  the assumption that $\nabla g$ is  $M$-Lipschitz continuous, the triangle inequality for norms,   definition of $v_f$ and $q_f$ in \eqref{def:vvff} and  \eqref{eq:def_qf}, respectively,
we conclude that 
\[
\|v_f \| - \|v\| \le \norm{v_f-v} = \norm{\nabla g(z_f) - \nabla g(z)} \le  M \norm{z_f-z} = \frac{M}{M+\lam^{-1}} \|q_f\| \le \|q_f\|
\]
and hence,  in view of \eqref{eq:inequality_vfg} and the inequality in  \eqref{inclusion_proxsolfg},  the inequality in \eqref{eq:incl_vff}  holds.
\end{proof}
\vgap

Lemma \ref{lem:approxsolff} assumes that the inclusion in \eqref{eq:ref4'} holds, or equivalently, that the function
$f$ defined in \eqref{eq:def_Lfzp} satisfies $(f+h)(u) \ge (f+h)(z) - \varepsilon$ for every $u$.
However, a close examination of its proof shows that the latter inequality is used only for $u=z_f$.

\vgap


{\bf  Proof of Proposition~\ref{prop:refapproxsol}}.  
First note that, since $\nabla g$ satisfies the second inequality in \eqref{ineq:concavity_g},  we see that  \eqref{uppercurvature_f} is satisfied with $f=g$ and $L=M$ (in particular $L=M+\lambda^{-1}$). Moreover, the elements defined in \eqref{eq:def_zg}, \eqref{eq:def_qg}, and \eqref{eq:def_deltag} correspond to \eqref{eq:def_zpz}, \eqref{eq:def_qf}, and \eqref{eq:def_deltaf}, respectively,  with $f$ replaced by $g$ and $L$ replaced by  $M+\lambda^{-1}$. Hence,  the inclusions in (a) and (b) as well as the first inequality in (b) follow immediately from \eqref{inclusion_qf}.
Now note that  the equality in (a) follows immediately from the definition of $q_g$ in \eqref{eq:def_qg}. Moreover,  the inequality in (b) implies the inequality in  (a). Hence,  let us proceed to prove the inequality in (b).
It follows from Lemma~\ref{lem:approxsolff} that the pair $(v,\delta_f)$ as in \eqref{def:vfg} and \eqref{eq:def_deltaf} satisfies the inclusion in  \eqref{eq:inclusion_vfg} and hence due to \eqref{argmin_qfdeltaf} with $(f,L)$ replaced by $(g,M+\lambda^{-1})$ and  \eqref{eq:inequality_vfg}, we have
\begin{align*}
\|q_g\|^2+2\left( M+\lam^{-1} \right)\delta_g\leq & \|v\|^2 + 2 \left( M+\lam^{-1} \right) \delta_f \le  \left  [\bar \rho + \sqrt{2(M+\lam^{-1})\varepsilon} \right ]^2,
\end{align*}
proving the second inequality in (b), and consequently concluding the proof of (a) and (b). Now to prove (c) 
first note that the inclusion follows immediately from the inclusion in (a) and the definition of $v_g$ in  \eqref{eq:def_vg}. On the other hand, the $M$-Lipschitz continuity of $\nabla g$ together with the definitions of  $q_g$ and  $v_g$ in \eqref{eq:def_qg} and \eqref{eq:def_vg}, respectively, and  the triangle inequality for norms imply that
\begin{equation*}
\|v_g\| \leq M\|z-z_g\| + \|q_g\|= \frac{M}{M+\lam^{-1}} \|q_g\| + \|q_g\|\leq 2\|q_g\|
\end{equation*}
which combined with  the inequality in (a) proves the inequality in (c). \hfill{ $\square$}

 
\section{Proof of Proposition~\ref{prop:gradient method}}\label{app:proofgradmet}

From the optimality condition for \eqref{gradientmethod}, we obtain
\begin{equation}\label{inc:appgradmet}
0\in \nabla g(z_{k-1})+\frac{z_k-z_{k-1}}{\lambda}+ \partial h(z_k). 
\end{equation}
Now let
 \begin{equation}\label{eq:psilambda}
 \Psi_\lambda=\Psi_{\lambda,k}:=g+\frac{1}{2\lambda}\|\cdot-z_{k-1}\|^2, \quad r_k:=\frac{z_{k-1}-z_k}{\lambda},
\end{equation} 
and note that $ \nabla \Psi_\lambda(z_{k-1})=\nabla g(z_{k-1})$, and  that $\Psi_\lambda$ is convex due to \eqref{ineq:concavity_g} and the assumption $\lambda <1/m$. Hence  Proposition~\ref{prop:transpForm}  yields
$
\nabla g(z_{k-1})=\nabla \Psi_\lambda(z_{k-1}) \in \partial_{\varepsilon_k} \Psi_\lambda(z_{k})
$
where  $\varepsilon_k=\Psi_\lambda (z_k)-\Psi_\lambda(z_{k-1})-\langle \nabla\Psi_\lam(z_{k-1}), z_k-z_{k-1}\rangle\geq 0$. 
The above inclusion combined with \eqref{inc:appgradmet} and  definition of $r_k$
imply that
$
r_k \in  \partial h(z_k)+\partial_{\varepsilon_k} \Psi_\lambda(z_k)\subset \partial_{\varepsilon_k} (h + \Psi_\lambda)(z_k)
$
where the last inclusion follows immediately from the definition of the operator $\partial_{\varepsilon_k}$ and convexity of $h$.
Hence, since  $(\tilde \varepsilon_k,\tilde v_k)=\lambda (\varepsilon_k,r_k)$ (see  \eqref{eq:statCGM} and  \eqref{eq:psilambda}), it follows from the above inclusion and the definition of $\Psi_\lambda$  that  the triple $(z_k,\tilde v_k,\tilde \varepsilon_k)$  satisfies the inclusion in \eqref{inclusion:GIPPF} with $\phi=g+h$ and $\lambda_k=\lambda$.

Now, to prove that the inequality in \eqref{eq:errocritGIPP} holds, first note that the definitions of $\varepsilon_k$ and $\Psi_\lambda$ together with property \eqref{ineq:concavity_g}, imply that 
$
\varepsilon_k\leq (\lambda M+1)\|z_{k-1}-z_k\|^2/(2\lambda).
$ Combining the latter inequality with the relations $\tilde v_k=z_{k-1}-z_k$ and $\tilde\varepsilon_k=\lambda\varepsilon_k$, we obtain
\begin{align*}
\|\tilde v_k\|^{2}+2\tilde \varepsilon_k&=\|z_{k-1}-z_k\|^2+2\lambda\varepsilon_k\leq \|z_{k-1}-z_k\|^2+(\lambda M+1)\|z_{k-1}-z_k\|^2\\[2mm]
&=(\lambda M+2)\|z_{k-1}-z_k\|^2= \frac{\lambda M+2}{4}\|z_{k-1}-z_k+\tilde v_k\|^2.
\end{align*}
Hence, since  $\lambda M <2$,  we conclude  that  $\sigma = (\lambda M+2)/4<1$ and  that   \eqref{eq:errocritGIPP} holds.  \hfill{ $\square$}




\end{document}